\documentclass[a4paper]{amsart}

\usepackage{latexsym}
\usepackage{a4wide}
\usepackage{amscd}
\usepackage{graphics}
\usepackage{amsmath}
\usepackage{amssymb}
\usepackage{mathrsfs}
\usepackage{enumerate}
\usepackage{hyperref}
\usepackage{pgf,tikz}
\usepackage[utf8]{inputenc}

\numberwithin{equation}{section} 

\newcommand{\R}{\mathbb R}

\newcommand{\N}{\mathbb N}

\newcommand{\eps}{\varepsilon}

\newcommand{\nor}[1]{\left\Vert #1 \right\Vert}

\newcommand{\dive }{\mathop{\rm div}}

\newtheorem{theorem}{Theorem}[section]
\newtheorem{corollary}[theorem]{Corollary}
\newtheorem{lemma}[theorem]{Lemma}
\newtheorem{proposition}[theorem]{Proposition}
 \theoremstyle{definition} 
 \newtheorem{definition}[theorem]{Definition}
\newtheorem{example}[theorem]{Example}

 \newtheorem{remark}[theorem]{Remark}

\definecolor{lilla}{rgb}{0.54, 0.17, 0.89}

\begin{document}

\title[]{On simple 
  eigenvalues of the fractional Laplacian under removal of small
  fractional capacity sets}

\author{Laura Abatangelo}
\address{Laura Abatangelo 
\newline \indent Dipartimento di Matematica e Applicazioni, 
Università degli Studi di Milano-Bicocca,
\newline \indent  Via Cozzi 55, 20125 Milano, Italy.}
\email{laura.abatangelo@unimib.it}

\author{Veronica Felli}
\address{Veronica Felli 
\newline \indent Dipartimento di Scienza dei Materiali, Università
degli Studi di Milano-Bicocca,
\newline \indent Via Cozzi 55, 20125 Milano, Italy.}
\email{veronica.felli@unimib.it}

\author{Benedetta Noris}
\address{Benedetta Noris
\newline \indent LAMFA: Laboratoire Amiénois de Mathématique
Fondamentale et Appliquée,
\newline \indent
UPJV Université de Picardie Jules Verne, 33 rue Saint-Leu, 80039 Amiens, France.}
\email{benedetta.noris@u-picardie.fr}

\date{February 19, 2019}

\thanks{L. Abatangelo and V. Felli are partially supported by the PRIN 2015
grant ``Variational methods, with applications to problems in
mathematical physics and geometry''  and the INDAM-GNAMPA
2018 grant ``Formula di monotonia e applicazioni: problemi frazionari
e stabilit\`a spettrale rispetto a perturbazioni del dominio''.
B. Noris is partially supported by the INDAM-GNAMPA group. 
All authors are partially supported by the project ERC Advanced Grant  2013 n. 339958: ``Complex Patterns for Strongly Interacting Dynamical Systems - COMPAT''
}

\begin{abstract}
  We consider the eigenvalue problem for the 
 restricted fractional Laplacian in a bounded domain with  homogeneous Dirichlet boundary conditions.  We introduce the notion of
  fractional capacity for compact subsets, with the property that the
  eigenvalues are not affected by the removal of zero fractional
  capacity sets.  Given a simple eigenvalue, we remove from the domain
  a family of compact sets which are concentrating to a set of zero
  fractional capacity and we detect the asymptotic expansion of the
  eigenvalue variation; this expansion depends on the eigenfunction
  associated to the limit eigenvalue.  Finally, we study the case in
  which the family of compact sets is concentrating to a point.
\end{abstract}

\maketitle


\medskip 

%
%

\section{Introduction}

In the present paper we consider the eigenvalue problem for the Dirichlet
fractional Laplacian in a bounded domain of $\R^N$. Our aim is to
provide asymptotic estimates of the eigenvalue variation when a small
vanishing 
set is removed.
In this context, the good notion of \emph{smallness}
  ensuring 
stability of the eigenvalue variation is related to the \emph{Gagliardo
fractional capacity}, which generalizes to the fractional setting the condenser capacity
appearing in the framework of the standard Laplace operator, see
Definition \ref{def:s-capacity} below.

In the classical setting of the Dirichlet Laplacian, Rauch and Taylor
\cite{rauch-taylor} observed that the spectrum does not change by
imposing homogeneous Dirichlet conditions on a compact polar subset,
i.e. on a subset of zero Newtonian capacity. Courtois \cite{courtois}
developed a perturbation theory for the Dirichlet spectrum of a
domain with small holes, with the capacity of holes
playing the role of a perturbation parameter. More precisely, in
\cite{courtois} it is proved that, if $K\subset \Omega$ is a compact
set, the $N$-th Dirichlet eigenvalue of the Laplacian in
$\Omega\setminus K$ is close to the $N$-th Dirichlet eigenvalue of the
Laplacian in $\Omega$ if (and only if) the capacity of the 
removed set $K$ in $\Omega$ is close to zero; furthermore, if the
capacity of $K$ is small, then the eigenvalue variation is even
differentiable with respect to the capacity of $K$ in $\Omega$. In
\cite{AFHL} asymptotic estimates for such eigenvalue variation were
obtained, highlighting a sharp relation between the order of vanishing
of an eigenfunction of the Dirichlet Laplacian at a point and the
leading term of the asymptotic expansion of the eigenvalue, as a
removed compact set concentrates at that point.  We also mention
\cite{Bertrand-Colbois,besson,Chavel-Feldman,flucher,Ozawa} for
related estimates of the eigenvalue variation for the Laplacian under removal of small
sets.

In order to formulate our problem, let us first introduce a suitable functional setting.
Let $\Omega\subset\R^N$, $N\geq1$, be an open set (bounded or unbounded). 
For $s\in(0,\min\{1,N/2\})$, 
we define the homogeneous fractional Sobolev space
  $\mathcal D^{s,2}(\Omega)$ as the completion of $C^\infty_c(\Omega)$
  with respect to the Gagliardo norm 
\[
[u]_{H^s(\R^N)}=\left(\int_{\R^N}\int_{\R^N}
  \frac{|u(x)-u(y)|^2}{|x-y|^{N+2s}} \,dx\,dy\right)^\frac{1}{2}. 
\]
We note that $\mathcal D^{s,2}(\Omega)\hookrightarrow \mathcal
D^{s,2}(\R^N)$ continuously by trivial extension. 
$\mathcal D^{s,2}(\Omega)$ is a Hilbert space with the scalar product 
\begin{equation}\label{eq:21}
  (u,v)_{\mathcal
    D^{s,2}(\Omega)}=\frac{C(N,s)}2\int_{\R^{N}}\int_{\R^{N}}\frac{(u(x)-u(y))(v(x)-v(y))}{|x-y|^{N+2s}}\,dx\,dy
  =\int_{\R^N}
  |\xi|^{2s}\overline{\widehat v(\xi) }\widehat u(\xi)\,d\xi,
\end{equation}
and the associated norm 
\[
\|u\|_{\mathcal D^{s,2}(\Omega)}=(u,u)_{\mathcal
    D^{s,2}(\Omega)}^{1/2}=\sqrt{\tfrac{C(N,s)}2}\,[u]_{H^s(\R^N)},
\]
where 
\begin{equation}\label{eq:14}
C(N,s)=\pi^{-\frac
  N2}2^{2s}\frac{\Gamma\big(\frac{N+2s}{2}\big)}{\Gamma(2-s)}s(1-s),
\end{equation}
$\Gamma$ is the Gamma function, and 
$\widehat{u}$ denotes the unitary Fourier transform of $u$.

We observe that, if $\Omega$ is bounded, then an equivalent norm on
$\mathcal D^{s,2}(\Omega)$ is 
\[
\|u\|_{L^2(\Omega)}+[u]_{H^s(\R^N)},
\]
see \cite[Corollary 5.2]{brasco-cinti}. As observed in
\cite{Brasco-Lindgren-Parini,brasco-salort}, in  general the space
$\mathcal D^{s,2}(\Omega)$ is smaller than the space $H^s_0(\Omega)$
defined as the closure of $C^\infty_c(\Omega)$
  with respect to the norm 
\[
\|u\|_{H^s(\Omega)}=
\|u\|_{L^2(\Omega)}
+[u]_{H^s(\Omega)}
\]
where 
\[
[u]_{H^s(\Omega)}=\left(
\int_\Omega\int_\Omega \frac{|u(x)-u(y)|^2}{|x-y|^{N+2s}} \,dx\,dy\right)^\frac{1}{2}.
\]
The two spaces $\mathcal D^{s,2}(\Omega)$ and $H^s_0(\Omega)$ 
coincide when $\Omega$ is a bounded Lipschitz open set and
$s\neq1/2$, see \cite[Proposition B.1]{Brasco-Lindgren-Parini}.
Furthermore, defining $H^s(\Omega)$ as the space $\left\{ u\in
  L^2(\Omega):[u]_{H^s(\Omega)}<+\infty\right\}$ endowed with the norm
$\|u\|_{H^s(\Omega)}=\|u\|_{L^2(\Omega)}+[u]_{H^s(\Omega)}$
and $\widetilde H^s(\Omega)$ as the space of  $H^s(\R^N)$-functions
that are zero in $\R^N\setminus\Omega$, it is known that, if $\Omega$ is
bounded and Lipschitz, then 
\[
H^s_0(\Omega)=\widetilde H^s(\Omega) \quad\text{if }s\neq\frac12
\]
and 
\[
H^s_0(\Omega)=\widetilde H^s(\Omega)=H^s(\Omega) \quad\text{if }s<\frac12,
\]
see \cite[Corollary 1.4.4.5]{grisvard}.

A key role in the perturbation theory we are going to develop for singularly perturbed
fractional eigenvalue problems is played by the \emph{Gagliardo
  fractional capacity}.

\begin{definition}\label{def:s-capacity}
Let $\Omega\subset\R^N$ be a bounded open set. Let $K\subset\Omega$ be
a compact set and let $\zeta_K\in C^\infty_c(\Omega)$ be such that
$\zeta_K=1$ in a neighborhood of $K$. For every $s\in
(0,\min\{1,N/2\})$, 
we define the \emph{Gagliardo $s$-fractional capacity} of $K$ in $\Omega$ as
\[
\mathop{\rm Cap}\nolimits^s_\Omega(K)=\inf \left\{
\|u\|_{\mathcal D^{s,2}(\Omega)}^2: \, u\in \mathcal D^{s,2}(\Omega)  \text{ and }  
u-\zeta_K\in \mathcal D^{s,2}(\Omega\setminus K) \right\}.
\]
\end{definition}
The Gagliardo $s$-capacity was introduced and studied in several
recent papers. We refer e.g. to \cite[Appendix A]{ritorto} for some
basic properties of the $s$-capacity; we also mention 
\cite{adams,Adams-Xiao,shi-xiao,Warma15} for some related notions of
fractional capacity.

From now on $\Omega\subset\R^N$ will denote a bounded open set.
We consider the following eigenvalue problem with
  homogeneous Dirichlet boundary conditions for the \emph{restricted fractional Laplacian}:
\begin{equation}\label{eq:eigenvalue_main}
\begin{cases}
(-\Delta)^s u=\lambda u, \quad& \text{in } \Omega, \\
u=0, &\text{in }\R^N\setminus\Omega.
\end{cases}
\end{equation}
We refer to Section \ref{sec:preliminaries} for a quick review of the definition and main properties of the fractional Laplacian $(-\Delta)^s$.
We say that $\lambda\in\R$ is an eigenvalue of problem
\eqref{eq:eigenvalue_main} 
if there exists some $u\in \mathcal
  D^{s,2}(\Omega)\setminus\{0\}$
 (called eigenfunction) such that
\[
 (u,v)_{\mathcal
    D^{s,2}(\Omega)}=\lambda \int_{\R^N} u(x)v(x) \,dx, \quad
\text{for all }  
v\in \mathcal D^{s,2}(\Omega).
\]
Since $(-\Delta)^s$ is a self-adjoint operator on
  $L^2(\Omega)$ with compact inverse, the Spectral Theorem implies
that the eigenvalues
 have finite multiplicity and form a diverging sequence
\[
0< \lambda_1^s(\Omega)\leq \lambda_2^s(\Omega) \leq \lambda_3^s(\Omega)\leq \ldots \to+\infty.
\]
We notice that, in contrast with the local case, a
  connectedness assumption on the domain $\Omega$ would lead to some
  loss of generality. Indeed, in the classical case the spectrum of the
  Dirichlet Laplacian in a disconnected domain is the union of the
  spectra on the connected components, whereas in the fractional case
  the spectrum is influenced by the mutual position of the connected
  components due to the nonlocal effects, see \cite[\S 2.3]{BP}.

We shall consider the eigenfunctions normalized as follows
\begin{equation}\label{eq:normalization}
\int_{\Omega}|u_j(x)|^2\,dx =1.
\end{equation}
Our first result
is the fractional counterpart of \cite[Theorem
  1.1]{courtois} and establishes
the continuity of the eigenvalue variation under
the removal of small fractional capacity sets. 
\begin{theorem}\label{thm:continuity}
Let $\Omega\subset\R^N$ be a bounded open set.
For $s\in (0,\min\{1,N/2\})$, $K\subset\Omega$ compact  and $k\in\N_*$, let
$\lambda_k^s(\Omega)$, respectively $\lambda_k^s(\Omega\setminus K)$, be the $k$-th eigenvalue of problem
\eqref{eq:eigenvalue_main} in
$\Omega$, respectively $\Omega\setminus K$. There exist $C>0$ and $\delta>0$
(independent of $K$) such that, if $\mathop{\rm
  Cap}\nolimits^s_\Omega(K)\leq\delta$, then 
\[
0\leq \lambda_k^s(\Omega\setminus K)-\lambda_k^s(\Omega)\leq C\left(\mathop{\rm Cap}\nolimits^s_\Omega(K)\right)^{1/2}.
\]
In particular we have that $\lambda_k^s(\Omega\setminus
K)\to\lambda_k^s(\Omega)$ as $\mathop{\rm Cap}\nolimits^s_\Omega(K)\to 0$.
\end{theorem}

Let us now consider a family of compact sets
concentrating to a set of zero capacity with the goal
  of detecting the leading term of the asymptotic expansion of the eigenvalue variation.

\begin{definition}\label{def:concentrates}
Let $\Omega\subset\R^N$ be a bounded open set.
Let $\{K_\varepsilon\}_{\varepsilon>0}$ be a family of compact sets contained in $\Omega$. 
We say that $K_\varepsilon$ is concentrating to a compact set $K\subset\Omega$ if for every open set $\omega$ such that $K\subset\omega\subseteq \Omega$ there exists $\varepsilon_\omega>0$ such that $K_\varepsilon\subset\omega$ for every $0<\varepsilon<\varepsilon_\omega$.
\end{definition}
We note that the limit set $K$ appearing in the previous definition
could be not unique. We comment on this definition in
Appendix \ref{sec:app_B}, where in particular we discuss the
relation between Definition \ref{def:concentrates} and the
classical notion of
convergence of sets in the sense of Mosco.

To state our main results in this direction, we need to introduce the
notion of fractional 
$u$-capacity for a function 
$u\in \mathcal D^{s,2}(\Omega)$.

\begin{definition}\label{def:U-capacity}
  Let $\Omega\subset\R^N$ be a bounded open set, $K\subset\Omega$ a
  compact set and $s\in (0,\min\{1,N/2\})$.  For every
  $u\in \mathcal D^{s,2}(\Omega)$, we define the \emph{$s$-fractional
    $u$-capacity} of $K$ in $\Omega$ as
\begin{equation}\label{eq:U_capacity_def}
\mathop{\rm Cap}\nolimits^s_\Omega(K,u)=
\inf \left\{
\|w\|_{\mathcal D^{s,2}(\Omega)}^2: \, w\in \mathcal D^{s,2}(\Omega)  \text{ and }  
w-u\in \mathcal D^{s,2}(\Omega\setminus K) \right\}.
\end{equation}
\end{definition}
More generally, we can define the fractional relative
$u$-capacity for every function $u\in
  H^s_{\rm loc}(\Omega)$. Indeed, letting $\zeta_K\in C^\infty_c(\Omega)$ be as
  in Definition \ref{def:s-capacity}, we have that $\zeta_Ku\in
  \mathcal{D}^{s,2}(\Omega)$, so that we can define
\begin{equation*}
\mathop{\rm Cap}\nolimits^s_\Omega(K,u)=
\inf \left\{
\|w\|_{\mathcal D^{s,2}(\Omega)}^2: \, w\in \mathcal D^{s,2}(\Omega)  \text{ and }  
w-\zeta_K u\in \mathcal D^{s,2}(\Omega\setminus K) \right\}.
\end{equation*}

The following theorem provides 
a sharp asymptotic expansion of the eigenvalue variation under removing
of a family
of compact sets concentrating to a zero fractional capacity set. 
In the classical setting of the Dirichlet Laplacian an analogous
result can be found in \cite[Theorem 1.4]{AFHL}, see also the proof of
\cite[Theorem 1.2]{courtois}.

\begin{theorem}\label{thm:expansion}
Let $\Omega\subset\R^N$ be a bounded open set.
For $s\in (0,\min\{1,N/2\})$ and $j\in\N_*$, let 
$\lambda_j^s(\Omega)$ be the $j$-th eigenvalue of
\eqref{eq:eigenvalue_main}.
Let $\{K_\varepsilon\}_{\varepsilon>0}$ be a family of compact sets contained in $\Omega$ concentrating to a compact set $K\subset\Omega$ in the sense of Definition \ref{def:concentrates}. If
\begin{equation*}
\lambda_j^s(\Omega)\text{ is simple} \qquad\text{and}\qquad 
\mathop{\rm Cap}\nolimits^s_\Omega(K)=0
\end{equation*}
then
\begin{equation}\label{eq:expansion_eigenv}
\lambda_j^s(\Omega\setminus K_\varepsilon)-\lambda_j^s(\Omega)=
\mathop{\rm Cap}\nolimits^s_\Omega(K_\varepsilon,u_j)
+o(\mathop{\rm Cap}\nolimits^s_\Omega(K_\varepsilon,u_j)),
\end{equation}
as $\varepsilon\to0^+$, where  $u_j \in \mathcal D^{s,2}(\Omega)$ is
an eigenfunction associated  to $\lambda_j^s(\Omega)$ normalized as in
\eqref{eq:normalization}. 
\end{theorem}

We can estimate the asymptotic behavior of the $s$-fractional
$u_j$-capacity as the family of compact sets $K_\varepsilon$
concentrates to a point, by exploiting some of the results in
\cite{FallFelli14}. Without loss of generality, we can assume that the
limit point is the origin, hence in the following we suppose that
$0\in\Omega$, with $\Omega$ being a bounded open set in $\R^N$. We
study the asymptotic behaviour of the quantity
$\mathop{\rm Cap}\nolimits^s_\Omega(K_\varepsilon,u_j)$ when
$K_\eps=\eps K$ for a given compact set $K\subset\R^N$ and
$\varepsilon\to0^+$. We observe that the family of compact sets
$\{\eps K\}_{\eps>0}$ concentrates (in the sense of Definition
\ref{def:concentrates}) to the singleton $\{0\}$, which has zero
$s$-capacity in $\Omega$ (see Example \ref{ex:singleton} ahead).

For $s\in (0,\min\{1,N/2\})$ and $j\in\N_*$, let $\lambda_j^s(\Omega)$
be the $j$-th eigenvalue of problem \eqref{eq:eigenvalue_main} and let
$u_j\in \mathcal D^{s,2}(\Omega)$ be 
an eigenfunction associated  to $\lambda_j^s(\Omega)$ normalized as in
\eqref{eq:normalization}.
In view of \cite{FallFelli14}, the asymptotic behavior of $u_j$
 at $0$
can be described in terms of 
the eigenvalues and the eigenfunctions of
 the following eigenvalue problem
\begin{align}\label{eq:eigsp}
  \begin{cases}
    -\dive\nolimits_{{\mathbb S}^{N}}(\theta_{N+1}^{1-2s}\nabla_{{\mathbb
        S}^{N}}\psi)=\mu\,
    \theta_{N+1}^{1-2s}\psi, &\text{in }{\mathbb S}^{N}_+,\\[5pt]
-\lim_{\theta_{N+1}\to 0^+} \theta_{N+1}^{1-2s}\nabla_{{\mathbb
    S}^{N}}\psi\cdot {\mathbf
  e}_{N+1}=0,&\text{on }\partial {\mathbb S}^{N}_+,
  \end{cases}
\end{align}
where ${\mathbb S}^{N}_+$ is the $N$-dimensional half-sphere
\begin{align*}
  &{\mathbb S}^{N}_+=\{(\theta_1,\theta_2,\dots, \theta_{N+1})\in
  {\mathbb S}^{N}:\theta_{N+1}>0\}=\left\{\tfrac{z}{|z|}:z\in \R^{N+1},\ z\cdot {\mathbf e}_{N+1}>0\right\},
\end{align*}
with ${\mathbf e}_{N+1}=(0,\dots,0,1)\in\R^{N+1}$. From classical
spectral theory,
problem \eqref{eq:eigsp} admits a diverging sequence of real eigenvalues
with finite multiplicity
\begin{equation*}
  \mu_1^s\leq\mu_2^s\leq\cdots\leq\mu_k^s\leq\cdots
\end{equation*}
Moreover $\mu_1^s=0$ and it is simple, i.e. $\mu_1^s<\mu_2^s$. We note
that, for $s=\frac12$, by reflection eigenfunctions of
\eqref{eq:eigsp} are spherical harmonics; then
$\{\mu_k^{1/2}:k\geq1\}=\{(N+k-2)(k-1):k\geq1\}$ and eigenfunctions
associated to the eigenvalue $(N+k-2)(k-1)$ are spherical harmonics of
degree $k-1$.

From \cite[Theorem 4.1 and Lemma 4.2]{FallFelli14} there exist $k_0\geq1$ and
$\psi\not\equiv0$ eigenfunction of problem \eqref{eq:eigsp} associated
to the eigenvalue $\mu_{k_0}^s$ such that, letting
\begin{equation}\label{eq:gamma_s_def}
\gamma_s=-\frac{N-2s}{2}+\sqrt{\left(\frac{N-2s}{2}\right)^2+\mu_{k_0}^s},
\end{equation}
it holds
\begin{equation}\label{eq:17}
\tilde u_\varepsilon(x):=\varepsilon^{-\gamma_s} u_j(\varepsilon x)\to 
\hat\psi(x)\quad\text{in } H^s(B'_R) \quad\text{as $\varepsilon\to 0^+$,}
\end{equation}
 for every $R>0$, where $B_R'=\{x\in \R^{N}: \, |x|<R\}$ and 
\begin{equation}\label{eq:hat-psi}
\hat\psi(x):=|x|^{\gamma_s}\psi\bigg(\frac{x}{|x|},0\bigg).
\end{equation}
We note that $\hat\psi\not\equiv 0$, see Section \ref{sec:restr-fract-lapl}.

\begin{theorem}\label{thm:cap_K_eps}
Let $\Omega\subset\R^N$ be a bounded open set with $0\in\Omega$ and $K\subset\Omega$ compact. For every $\varepsilon>0$ let $K_\varepsilon=\varepsilon K$.
For $s\in (0,\min\{1,N/2\})$ and $j\in\N_*$, let $\lambda_j^s(\Omega)$
be the $j$-th eigenvalue of problem \eqref{eq:eigenvalue_main} and let
$u_j\in \mathcal D^{s,2}(\Omega)$ be 
an eigenfunction associated  to $\lambda_j^s(\Omega)$ normalized as in
\eqref{eq:normalization}. 
Then, as $\varepsilon\to0^+$, it holds
\begin{equation}\label{eq:cap_K_eps}
  \mathop{\rm Cap}\nolimits^s_\Omega(K_\varepsilon, u_j)
=\varepsilon^{N+2(\gamma_s-s)}\left\{
\mathop{\rm Cap}\nolimits^s_{\R^N}(K,\hat\psi)
+o(1)\right\},
\end{equation}
with $\gamma_s$ and $\hat\psi$ as in \eqref{eq:gamma_s_def} and \eqref{eq:hat-psi} respectively.
\end{theorem}

As a consequence of Theorems \ref{thm:expansion} and \ref{thm:cap_K_eps}, we deduce the following.

\begin{theorem}\label{thm:main}
Let $\Omega\subset\R^N$ be a bounded open set with $0\in\Omega$ and $K\subset\Omega$ compact. For every $\varepsilon>0$ let $K_\varepsilon=\varepsilon K$.
For $s\in (0,\min\{1,N/2\})$ and $j\in\N_*$, 
 let $\lambda_j^s(\Omega)$
be the $j$-th eigenvalue of problem \eqref{eq:eigenvalue_main} and let
$u_j\in \mathcal D^{s,2}(\Omega)$ be 
an associated eigenfunction satisfying \eqref{eq:normalization}.
If $\lambda_j^s(\Omega)$ is simple, then, as $\varepsilon\to0^+$, it holds
\begin{equation}\label{eq:main_thm}
\lambda_j^s(\Omega\setminus K_\varepsilon)-\lambda_j^s(\Omega)
=\varepsilon^{N+2(\gamma_s-s)} 
\mathop{\rm Cap}\nolimits^s_{\R^N}(K,\hat\psi)
 + o(\varepsilon^{N+2(\gamma_s-s)}),
\end{equation}
with $\gamma_s$ and $\hat\psi$ as in \eqref{eq:gamma_s_def} and \eqref{eq:hat-psi} respectively.
\end{theorem}

The asymptotic expansion \eqref{eq:main_thm} is sharp 
  whenever $\mathop{\rm Cap}\nolimits^s_{\R^N}(K,\hat\psi)\neq0$,
for example when $K$ has nonzero Lebesgue measure in $\R^N$, 
as observed in Corollary \ref{cor:sharp_asymp} below. 
 We mention that the fractional capacity
  $\mathop{\rm Cap}\nolimits^s_{\R^N}(K,\hat\psi)$ on the whole
  $\R^N$ appearing in the leading term of the expansion
  \eqref{eq:main_thm} is related to the weighted capacity of $K$ in
  $\R^{N+1}$ with respect to the Muckenhoupt weight $|t|^{1-2s}$, see
  Remark \ref{rem:heinonen}; we refer to \cite[Chapter
  2]{heinonen2012nonlinear} for a discussion on the properties
of such capacity.

\begin{corollary}\label{cor:sharp_asymp}
Under the same assumptions as in Theorem \ref{thm:main}, suppose moreover that the $N$-dimensional Lebesgue measure of $K$ is strictly positive. Then
\begin{equation}
\lim_{\varepsilon\to0^+}\frac{\lambda_j^s(\Omega\setminus K_\varepsilon)-\lambda_j^s(\Omega)}{\varepsilon^{N+2(\gamma_s-s)}}
= \mathop{\rm Cap}\nolimits^s_{\R^N}(K,\hat\psi)>0.
\end{equation}
\end{corollary}

\begin{remark} It is worth mentioning that in the
    literature, besides the notion of \emph{restricted fractional
      Laplacian} treated in the present paper, also the so called
    \emph{spectral fractional Laplacian} (defined as the power of
    $-\Delta$ obtained by using its spectral decomposition) is often
    taken into consideration. The restricted and the spectral
    fractional Laplacians on bounded domains are different operators,
    as observed in \cite{musina-nazarov} and \cite{SV}. The problem of spectral stability
    investigated in the present paper turns out to be much simpler for
    the spectral fractional Laplacian than for the restricted one,
    since the eigenvalues of the spectral fractional $s$-Laplacian are
    just the $s$-power of the eigenvalues of the classical Dirichlet
    Laplacian; hence the asymptotics of eigenvalues under removal of
    small sets can be easily deduced from the classical case treated
    in \cite{AFHL}.

Denoting as $\{\lambda_j(\Omega)\}_{j=1}^\infty$  the eigenvalues the Laplacian
in a bounded open set $\Omega\subset\R^N$ with homogeneous boundary
conditions and by  $\varphi_j$ the eigenfunction associated to
$\lambda_j(\Omega)$ normalized with respect to the $L^2(\Omega)$-norm,
 the spectral fractional Laplacian with homogeneous Dirichlet boundary
 conditions can be defined, for all $s\in(0,1)$, as
\[
(-\Delta_{\rm spectral})^s u(x)=\sum_{j=1}^{+\infty} (\lambda_j(\Omega))^s \left(\int_\Omega u\varphi_j\,dx\right) \varphi_j(x), \quad x\in \Omega.
\]
The eigenvalues and the eigenfunctions of $(-\Delta_{\rm spectral})^s$ 
are, respectively, $\nu_j^s(\Omega):=(\lambda_j(\Omega))^s$ and $\varphi_j$.
Then, from \cite[Theorem 1.4]{AFHL} it follows easily that, 
if $\lambda_j(\Omega)$ is simple and 
$\{K_\varepsilon\}_{\varepsilon>0}$ is a family of compact sets
contained in $\Omega$ concentrating to a null capacity compact set,
then 
\begin{equation*}
\nu_j^s(\Omega\setminus K_\varepsilon)-\nu_j^s(\Omega)=s
(\lambda_j(\Omega))^{s-1}\mathop{\rm Cap}\nolimits_\Omega(K_\varepsilon,\varphi_j)
+o(\mathop{\rm Cap}\nolimits_\Omega(K_\varepsilon,\varphi_j)),
\end{equation*}
as $\varepsilon\to0^+$, where
$\mbox{\rm Cap}_{\Omega}(K_\varepsilon,\varphi_j)=
\inf\left\{\int_{\Omega}|\nabla f|^2:\ f\in H^1_0(\Omega)\text{ and
  }f-\varphi_j\in H^1_0(\Omega\setminus K_\varepsilon) \right\}$.
Asymptotic expansions of
$\mbox{\rm Cap}_{\Omega}(K_\varepsilon,\varphi_j)$ are obtained in
\cite{AFHL} in several situations.  

Comparing the above asymptotic expansion 
for the spectral fractional Laplacian with the expansion derived in
Theorem \ref{thm:main}, we note that only in the case of the restricted
fractional Laplacian the vanishing order of the eigenvalue variation
depends on the power $s$; hence the eigenvalues of the two operators exhibit quite
different asymptotic behaviours under removal of small sets.
\end{remark}

The paper is organized as follows. In Section \ref{sec:preliminaries}
we collect some preliminary results.  In Sections \ref{sec:continuity}
and \ref{sec:expansion} we prove respectively Theorems
\ref{thm:continuity} and \ref{thm:expansion}.  In Section
\ref{sec:scaling} we present the proofs of Theorems
\ref{thm:cap_K_eps}, \ref{thm:main} and of Corollary \ref{cor:sharp_asymp}.
Finally, in Appendix \ref{sec:app_A} we prove an $L^\infty$ bound for
eigenfunctions which is needed in Section \ref{sec:continuity} and in
Appendix \ref{sec:app_B} we discuss the Definition
\ref{def:concentrates} of concentrating compact sets.

\section{Preliminaries}\label{sec:preliminaries}
In this section we recall some known facts and present some  preliminary results.

\subsection{Restricted fractional
      Laplacian and Caffarelli-Silvestre extension}\label{sec:restr-fract-lapl}
The fractional Laplacian $(-\Delta)^s$ can be defined over the space $C^\infty_c(\R^N)$ by the principal value integral
\[
(-\Delta)^s u(x)=C(N,s) \lim_{\eps\to0^+} \int_{\R^N\setminus B_\eps(x)} 
\frac{u(x)-u(y)}{|x-y|^{N+2s}} \,dy,
\]
where $C(N,s)$ is given in \eqref{eq:14}, 
or equivalently through the Fourier transform:
\[
\mathcal{F}((-\Delta)^s u)(\xi)=|\xi|^{2s} \mathcal{F}u(\xi), \qquad \xi\in \R^N.
\] 
The scalar product of $\mathcal D^{s,2}(\R^N)$ defined in
\eqref{eq:21} is  naturally associated to $(-\Delta)^s$, in the sense
that $(-\Delta)^s$ can be extended to a bounded linear operator from 
$\mathcal D^{s,2}(\R^N)$ to its dual $(\mathcal{D}^{s,2}(\R^N) )^*$,
which actually coincides with the Riesz isomorphism of
$\mathcal{D}^{s,2}(\R^N)$ with respect to the scalar product
\eqref{eq:21}, i.e. 
\[
 \phantom{a}_{( \mathcal{D}^{s,2}(\R^N) )^*}\langle(-\Delta)^s u, v\rangle_{ \mathcal{D}^{s,2}(\R^N) } = (u,v)_{\mathcal{D}^{s,2}(\R^N)} 
\]
for all $u,v\in \mathcal{D}^{s,2}(\R^N)$.

In \cite{CaffarelliSilvestre07} Caffarelli and
  Silvestre proved that $(-\Delta)^s$ can be realized as a
  Dirichlet-to-Neumann operator, i.e. as an operator mapping a
  Dirichlet boundary condition to a Neumann condition via an extension
  problem on the half space
\[  
\R^{N+1}_+=\{ (x,t)\in\R^{N+1}:\, x\in\R^N, t>0 \}.
\]
For every
  $U,V\in C^\infty_c(\overline{\R^{N+1}_+})$, let
\begin{equation*}
q(U,V)=\int_{\R^{N+1}_+}t^{1-2s}\nabla U(x,t)\cdot\nabla V(x,t) \,dx\,dt.
\end{equation*}
We define $\mathcal{D}^{1,2}(\R^{N+1}_+;t^{1-2s})$ as the completion of $C^\infty_c(\overline{\R^{N+1}_+})$ with respect to the norm
\[
\|U\|_{\mathcal{D}^{1,2}(\R^{N+1}_+;t^{1-2s})}= \sqrt{q(U,U)}.
\]
There exists a well-defined continuous trace map 
\begin{equation}\label{eq:trace}
\mathop{\rm Tr}: \mathcal{D}^{1,2}(\R^{N+1}_+;t^{1-2s}) \to \mathcal{D}^{s,2}(\R^N)
\end{equation}
which is onto (see for example \cite{brandle2013concave}).
By the Caffarelli-Silvestre extension theorem
\cite{CaffarelliSilvestre07}, given $u\in \mathcal D^{s,2}(\R^N)$, the minimization problem
\[
\min\left\{ q(W,W) : \, W\in \mathcal{D}^{1,2}(\R^{N+1}_+;t^{1-2s}),\, \mathop{\rm Tr}W=u\right\}
\]
admits a unique minimizer $U=\mathcal H(u) \in
\mathcal{D}^{1,2}(\R^{N+1}_+;t^{1-2s})$, which moreover
satisfies 
\begin{equation}\label{eq:13}
q(\mathcal H(u),W)=\kappa_s(u,\mathop{\rm Tr}W)_{\mathcal
  D^{s,2}(\R^N)},\quad\text{for all }\varphi \in
\mathcal{D}^{1,2}(\R^{N+1}_+;t^{1-2s}),
\end{equation}
where
\[
\kappa_s=\frac{\Gamma(1-s)}{2^{2s-1}\Gamma(s)},
\]
i.e. $U=\mathcal H(u)$
 weakly solves
\begin{equation*}
\begin{cases}
-\text{div}(t^{1-2s} \nabla U)=0, \quad&\text{in } \R^{N+1}_+, \\
\lim_{t\to0^+}\left(-t^{1-2s} \partial_t U \right)= \kappa_s
(-\Delta)^su, &\text{in } \R^N\times\{0\}.
\end{cases}
\end{equation*}
From \eqref{eq:13} it follows that
\begin{equation}\label{eq:4}
  \|U\|_{\mathcal{D}^{1,2}(\R^{N+1}_+;t^{1-2s})}^2=\kappa_s\|u\|_{\mathcal D^{s,2}(\R^N)}^2.
\end{equation}
As a consequence, if $\lambda_j^s(\Omega)$ is an eigenvalue of
\eqref{eq:eigenvalue_main} for a certain $j\in\N_*=\N\setminus\{0\}$
and $u_j \in \mathcal D^{s,2}(\Omega)$ is an associated eigenfunction, the extension $U_j=\mathcal H(u_j)$ satisfies $\mathop{\rm Tr} U_j=u_j$ and
\begin{equation}\label{eq:realization_s}
\begin{cases}
-\text{div}(t^{1-2s} \nabla U_j)=0, \quad&\text{in } \R^{N+1}_+, \\
\lim_{t\to0^+}\left(-t^{1-2s} \partial_t U_j
\right)=\lambda_j^s(\Omega) \kappa_s \mathop{\rm Tr} U_j,
 &\text{in } \Omega\times\{0\}, \\
U_j=0, &\text{in } (\R^N\setminus \Omega)\times\{0\},
\end{cases}
\end{equation}
in a weak sense, that is
\begin{equation}\label{eq:realization_s_w}
\begin{cases}
U_j\in
  \mathcal{D}^{1,2}_{\Omega^c}(\R^{N+1}_+;t^{1-2s}),\\
{\displaystyle{ q(U_j,\phi)
=\lambda_j^s(\Omega) \kappa_s\int_\Omega \mathop{\rm Tr}U_j\mathop{\rm Tr}\phi\,dx}}\quad\text{for every }\phi\in
\mathcal{D}^{1,2}_{\Omega^c}(\R^{N+1}_+;t^{1-2s}).
\end{cases}
\end{equation}
Here, the space $\mathcal{D}^{1,2}_{\Omega^c}(\R^{N+1}_+;t^{1-2s})$ is
defined as the
closure of $C^\infty_c(\R^{N+1}_+\cup\Omega)$ in $\mathcal{D}^{1,2}(\R^{N+1}_+;t^{1-2s})$;
we also have the equivalent characterization
\begin{equation}\label{eq:15}
\mathcal{D}^{1,2}_{\Omega^c}(\R^{N+1}_+;t^{1-2s})
=\{U\in \mathcal{D}^{1,2}(\R^{N+1}_+;t^{1-2s}): \, 
\mathop{\rm Tr}U\in \mathcal D^{s,2}(\Omega)\}.
\end{equation}
We can consider equivalently either \eqref{eq:realization_s_w} or \eqref{eq:eigenvalue_main} with $\lambda=\lambda_j^s(\Omega)$.
In this extended setting, the eigenvalues admit the following Courant-Fisher minimax characterization 
\begin{equation}\label{eq:minimax}
\lambda_j^s(\Omega)=\min_{\mathcal U\in\mathcal S_j}\, 
 \max_{\substack{U\in \mathcal{U} \\
      \|\mathop{\rm Tr}U\|_{L^2(\Omega)}\neq0}} \!\!\!
\mathcal{R}(U)
\end{equation}
where $\mathcal S_j$ denotes the family of all $j$-dimensional 
subspaces of $\mathcal{D}^{1,2}_{\Omega^c}(\R^{N+1}_+;t^{1-2s})$ and 
$\mathcal R$ is the Rayleigh type quotient defined as
\begin{equation}\label{eq:rayleigh_s}
\mathcal{R}(U)=\frac{q(U,U)}{\kappa_s \int_{\Omega} |\mathop{\rm Tr}U(x)|^2\,dx}.
\end{equation}

\begin{remark}\label{rem:capacity}
If $\Omega\subset\R^N$ is  bounded and open  and $K\subset\Omega$ is a
compact subset, in view of the Caffarelli-Silvestre 
extension result described above and, in particular, of \eqref{eq:4},
we can characterize the  
Gagliardo $s$-fractional capacity introduced in Definition
\ref{def:s-capacity} as 
\begin{equation*}
\mathop{\rm Cap}\nolimits^s_\Omega(K)=\frac1{\kappa_s}\inf \left\{
q(W,W): \, W\in \mathcal{D}^{1,2}_{\Omega^c}(\R^{N+1}_+;t^{1-2s}) \text{ and } 
W-\eta_K \in \mathcal{D}^{1,2}_{\Omega^c\cup K}(\R^{N+1}_+;t^{1-2s}) 
\right\},
\end{equation*}
where $\eta_K \in C^\infty_c(\R^{N+1}_+\cup\Omega)$ is any fixed
function such that $\eta_K=1$ in a neighborhood of $K$.

Correspondingly, for any $u\in \mathcal D^{s,2}(\Omega)$, we can characterize the  $s$-fractional
$u$-capacity of $K$ in $\Omega$  introduced in Definition \ref{def:U-capacity}
as 
\begin{equation}\label{eq:16}
\mathop{\rm Cap}\nolimits^s_\Omega(K,u)=\tfrac1{\kappa_s}\inf \left\{
q(W,W):  W\in \mathcal{D}^{1,2}_{\Omega^c}(\R^{N+1}_+;t^{1-2s}),\,
W-U \in \mathcal{D}^{1,2}_{\Omega^c\cup K}(\R^{N+1}_+;t^{1-2s}) 
\right\}\!\!
\end{equation}
where $U\in  \mathcal{D}^{1,2}_{\Omega^c}(\R^{N+1}_+;t^{1-2s})$ is
such that $\mathop{\rm Tr}U=u$.
\end{remark}

\subsection{Local asymptotic behaviour of eigenfunctions and their extension}\label{sec:local-asympt-behav}
For  $j\in\N_*$ and $s\in (0,\min\{1,\frac N2\})$, let $\lambda_j^s(\Omega)$
be the $j$-th eigenvalue of problem \eqref{eq:eigenvalue_main} and let
$U_j\in
 \mathcal{D}^{1,2}_{\Omega^c}(\R^{N+1}_+;t^{1-2s})$ be a solution to 
 \eqref{eq:realization_s} such that its
 trace $u_j=\mathop{\rm Tr}U_j$ satisfies the normalization condition 
\eqref{eq:normalization}.
In \cite{FallFelli14}, the asymptotic behavior of $U_j$ (and
consequently of its trace $u_j$) at $0$ has been described in terms of 
the eigenvalues and the eigenfunctions of problem
\eqref{eq:eigsp}. More precisely, in \cite[Theorem 4.1 and Lemma
4.2]{FallFelli14}
it has been proved that there exist $k_0\geq1$ and
$\psi\not\equiv0$ eigenfunction of problem \eqref{eq:eigsp} associated
to the eigenvalue $\mu_{k_0}^s$ such that
\begin{equation}\label{eq:12}
\tilde U_\varepsilon(z):=\varepsilon^{-\gamma_s} U_j(\varepsilon z)\to 
\tilde\psi(z):= |z|^{\gamma_s}
\psi\left(\frac{z}{|z|}\right)\quad\text{in } H^1(B_R^+;t^{1-2s}) 
\quad\text{as $\varepsilon\to 0^+$,}
\end{equation}
 for every $R>0$, where
  $B_R^+=\{z=(x,t)\in \R^{N+1}_+: \, |z|<R\}$,
 $\gamma_s$ is given in \eqref{eq:gamma_s_def}, and the space $H^1(B_R^+;t^{1-2s})$
is defined in Section \ref{sec:sobolev-hardy-type} below.

 The
convergence \eqref{eq:17} stated in the introduction follows from 
\eqref{eq:12} by passing to the traces.
\begin{remark}\label{rem:hatpsinonzero}
We note that the limit profile $\hat\psi:=\mathop{\rm Tr}\tilde\psi$
appearing in \eqref{eq:17} is not identically null; indeed 
$\tilde\psi$ and $t^{1-2s} \partial_t
\tilde\psi$ can not both vanish on $\partial \R^{N+1}_+$, because
otherwise $\tilde\psi$  would be
    a weak solution to the equation $\text{div}(t^{1-2s} \nabla
    \tilde\psi)=0$ satisfying both Dirichlet and weighted Neumann
    homogeneous boundary conditions and its trivial extension in
    $\R^{N+1}$ would violate the  unique continuation
    principle for elliptic equations with Muckenhoupt weights proved
    in \cite{tao-zhang} (see also \cite{GL}, and
    \cite[Proposition 2.2]{Ruland}).
  \end{remark}

\subsection{Sobolev and Hardy-type inequalities}\label{sec:sobolev-hardy-type}
For every $s\in \big(0,\min\{1,\frac N2\}\big)$ (so that $N-2s>0$), let
\begin{equation}\label{eq:2*s}
2^*\!(s)=\frac{2N}{N-2s}.
\end{equation}
The following Sobolev inequalities and compactness results can be found for example in \cite{DNPV12}. 

\begin{theorem}[{\cite[Theorems 6.5 and 6.7, Corollary 7.2]{DNPV12}}]\label{thm:immersions}
Let $\Omega\subset\R^N$, $N\geq 1$, be a bounded, open set of class $C^{0,1}$ and let $s\in (0,\min\{1,N/2\})$.
\begin{itemize}
\item[(i)] There exists a positive constant $S_{N,s}$ such that 
\[
S_{N,s}\|u\|_{L^{2^*(s)}(\R^N)}\leq \|u\|_{\mathcal D^{s,2}(\R^N)}
\quad\text{for all }u\in \mathcal D^{s,2}(\R^N).
\]
\item[(ii)] There exists a positive constant $C=C(N,s,\Omega)$ such that for every $u\in H^s(\Omega)$ and for every $q\in [1,2^*\!(s)]$ it holds
\[
\|u\|_{L^q(\Omega)} \leq C\|u\|_{H^s(\Omega)}.
\]
\item[(iii)] If $\mathcal{I}$ is a bounded subset of $H^s(\Omega)$, then $\mathcal{I}$ is pre-compact in $L^q(\Omega)$ for every $q\in [1,2^*\!(s))$.
\end{itemize}
\end{theorem}

Let us recall some fractional  Hardy-type inequalities. 
For any $s\in(0,1)$, the following Hardy-type inequality for
$\mathcal{D}^{s,2}(\R^N)$-functions was established in \cite{Herbst77}:
\begin{equation}\label{eq:Herbst}
\Lambda_{N,s} \int_{\R^N} \frac{u^2(x)}{|x|^{2s}}\,dx \leq  \|u\|_{\mathcal{D}^{s,2}(\R^N)}^2
\qquad \text{for all } u \in \mathcal{D}^{s,2}(\R^N),
\end{equation}
where
\[
\Lambda_{N,s}=2^{2s} \frac{\Gamma^2\left(\frac{N+2s}{4}\right)}{\Gamma^2\left(\frac{N-2s}{4}\right)}.
\]
By combining the \eqref{eq:Herbst} and
 \eqref{eq:4}, we obtain the following Hardy-trace inequality: 
\begin{align}\label{eq:half_space_hardy}
   \Lambda_{N,s}\kappa_s\int_{\R^N}
   \frac{|\mathop{\rm Tr}U|^2}{|x|^{2s}}\,dx\leq \int_{\R^{N+1}_+}t^{1-2s}|\nabla
   U|^2\,dx\,dt,\quad\text{for all }U\in{\mathcal
     D}^{1,2}({\R^{N+1}_+};t^{1-2s}).
 \end{align}
Relation \eqref{eq:half_space_hardy} implies in particular that, if $\Omega$ is bounded,
\begin{equation}\label{eq:hardy_bdd}
   \int_{\Omega}
|\mathop{\rm Tr}U|^2\,dx\leq \frac{\text{diam}(\Omega)^{2s}}{\Lambda_{N,s}\kappa_s} \|U\|^2_{\mathcal{D}^{1,2}({\R^{N+1}_+};t^{1-2s})},
   \quad\text{for all }U\in{\mathcal
     D}^{1,2}_{\Omega^c}({\R^{N+1}_+};t^{1-2s}),
\end{equation}
where $\text{diam}(\Omega)$ is the diameter of $\Omega$.

For $r>0$, let $B_r^+=\{z=(x,t)\in \R^{N+1}_+: \, |z|<r\}$. We define $H^1(B_r^+;t^{1-2s})$ as the completion of $C^\infty(\overline{B_r^+})$ with respect to
\[
\|U\|_{H^1(B_r^+;t^{1-2s})} =\left( \int_{B_r^+} t^{1-2s} (|\nabla U|^2+U^2) \,dx\,dt \right)^{1/2}.
\]
The following Hardy type inequality with boundary terms was proved in \cite{FallFelli14}.

\begin{lemma}[{\cite[Lemma 2.4]{FallFelli14}}]\label{lem:Hardy}
Let $s\in (0,\min\{1,N/2\})$.
For all $r > 0$ and $U \in H^1(B_r^+;t^{1-2s})$, the following holds
\begin{equation*}
\left(\frac{N-2s}{2}\right)^2 \int_{B_r^+} t^{1-2s} \frac{U^2(z)}{|z|^2} \,dz \leq
\int_{B_r^+} t^{1-2s} \left( \nabla U(z)\cdot \frac{z}{|z|} \right)^2 \,dz+
\left(\frac{N-2s}{2r}\right) \int_{S_r^+} t^{1-2s} U^2(z) dS,
\end{equation*}
where $S_r^+=\{z=(x,t)\in \R^{N+1}_+: \, |z|=r\}$ and $dS$ denotes the volume element on $S_r^+$.
\end{lemma}

As a particular case of the inequality stated in Lemma
\ref{lem:Hardy}, we obtain the following
\begin{equation}\label{eq:8}
\left(\frac{N-2s}{2}\right)^2 \int_{\R^{N+1}_+} t^{1-2s} \frac{U^2(z)}{|z|^2} \,dz \leq
\int_{\R^{N+1}_+}t^{1-2s} | \nabla U(z)|^2 \,dz,
\end{equation}
for all $U\in \mathcal{D}^{1,2}(\R^{N+1}_+;t^{1-2s})$ and $s\in (0,\min\{1,N/2\})$. 

\subsection{Fractional capacities and capacitary potentials}
We observe that, by Stampacchia's Theorem, the infimum in Remark
\ref{rem:capacity} is achieved by a unique function
$V_{\Omega,K} \in
\mathcal{D}^{1,2}_{\Omega^c}(\R^{N+1}_+;t^{1-2s}) $,
with
$V_{\Omega,K}-\eta_K \in
\mathcal{D}^{1,2}_{\Omega^c\cup K}(\R^{N+1}_+;t^{1-2s}) $, so that
\begin{equation}\label{eq:grad_V_k}
\mathop{\rm Cap}\nolimits^s_\Omega(K)=\frac1{\kappa_s} q(V_{\Omega,K} ,V_{\Omega,K});
\end{equation}
moreover $V_{\Omega,K}$ satisfies
\[
q(V_{\Omega,K},v-V_{\Omega,K})\geq 0
\]
for all $v\in \mathcal{D}^{1,2}_{\Omega^c}(\R^{N+1}_+;t^{1-2s})$ with $v-\eta_K \in \mathcal{D}^{1,2}_{\Omega^c\cup K}(\R^{N+1}_+;t^{1-2s})$.
Equivalently, we have that $V_{\Omega,K}\in \mathcal{D}^{1,2}_{\Omega^c}(\R^{N+1}_+;t^{1-2s})$ is the unique function such that $V_{\Omega,K}-\eta_K \in \mathcal{D}^{1,2}_{\Omega^c\cup K}(\R^{N+1}_+;t^{1-2s})$ and
\begin{equation}\label{eq:V_K_weak}
q(V_{\Omega,K},\phi )=0 \qquad\text{for all } \phi\in \mathcal{D}^{1,2}_{\Omega^c\cup K}(\R^{N+1}_+;t^{1-2s}),
\end{equation}
that is to say, $V_{\Omega,K}$ is the unique weak solution of
\begin{equation}\label{eq:V_K}
\begin{cases}
-\text{div}(t^{1-2s}\nabla V_{\Omega,K})=0, \quad&\text{in } \R^{N+1}_+, \\
\lim_{t\to0^+}\left(-t^{1-2s} \partial_t V_{\Omega,K} \right)=0, &\text{in }
(\Omega\setminus K)\times\{0\}, \\
V_{\Omega,K}=0, &\text{in } (\R^N\setminus\Omega)\times\{0\}, \\
V_{\Omega,K}=1, &\text{in } K\times\{0\}.
\end{cases}
\end{equation}
We also observe that $\mathop{\rm Tr}V_{\Omega,K}$
  attains the infimum in Definition \ref{def:s-capacity}.

Since $V_{\Omega,K}^-$ and $(V_{\Omega,K}-1)^+$ belong to
  $\mathcal{D}^{1,2}_{\Omega^c\cup K}(\R^{N+1}_+;t^{1-2s})$, we can
  choose $\phi=V_{\Omega,K}^-$ and $\phi=(V_{\Omega,K}-1)^+$ in \eqref{eq:V_K_weak}; in
  this way we obtain that $V_{\Omega,K}^-=(V_{\Omega,K}-1)^+\equiv0$, that is
\begin{equation}\label{eq:10}
0\leq V_{\Omega,K}\leq 1\quad\text{a.e. in }\R^{N+1}_+.
\end{equation}

  \begin{example}[\bf Capacity of a point]\label{ex:singleton}
    If $\Omega\subset\R^N$ is an open set, $s\in
    (0,\min\{1,N/2\})$, and $P\in\Omega$, then  
 \begin{equation}\label{eq:11}
\mathop{\rm Cap}\nolimits^s_\Omega(\{P\})=0.
\end{equation}
Indeed, for every $n\in\N_*$, let $W_n\in
C^\infty(\R^{N+1})$ be such that 
$W_n(z)=1$ for $|z-P|\leq\frac1n$, 
$W_n(z)=0$ for $|z-P|\geq\frac2n$, and $|\nabla W_n(z)|\leq 2n$ for all
$z\in\R^{N+1}$.  
Then, for $n$ sufficiently large,  the restriction
$W_n\big|_{\R^{N+1}_+}$ belongs to
$\mathcal{D}^{1,2}_{\Omega^c}(\R^{N+1}_+;t^{1-2s})$ and is equal to
$1$ in a neighborhood of $\{P\}$. Moreover
\[
q(W_n,W_n)\leq {\rm const\,}n^2\int_{1/n}^{2/n}r^{N+1-2s}\,dr
=O(n^{2s-N})=o(1)\quad\text{as }n\to+\infty,
\]
thus proving  \eqref{eq:11}.
\end{example}

In order to prove that the spectrum of restricted fractional $s$-Laplacian
in $\Omega$  does not change by
removing a subset of zero fractional $s$-capacity, the following
result is needed.

\begin{proposition}\label{prop:cap0}
Let $\Omega\subset\R^N$ be an open set, $K\subset\Omega$  compact and $s\in (0,\min\{1,N/2\})$.
The following three assertions are equivalent:
\begin{itemize}
\item[(i)] $\mathop{\rm Cap}\nolimits^s_\Omega(K)=0$;
\item[(ii)] $\mathcal{D}^{1,2}_{\Omega^c}(\R^{N+1}_+;t^{1-2s})=\mathcal{D}^{1,2}_{\Omega^c\cup K}(\R^{N+1}_+;t^{1-2s})$;
\item[(iii)]
  $\mathcal D^{s,2}(\Omega)=\mathcal D^{s,2}(\Omega\setminus K)$.
\end{itemize}
\end{proposition}
\begin{proof}
  It will be sufficient to prove that (i) is equivalent to (ii), since
  then the equivalence of (iii) follows from the fact that the
  restriction to $\Omega$ of the trace map $\mathop{\rm Tr}$ defined
  in \eqref{eq:trace} is onto and the characterization
    of spaces given in \eqref{eq:15}.

Suppose first that $\mathcal{D}^{1,2}_{\Omega^c}(\R^{N+1}_+;t^{1-2s})=\mathcal{D}^{1,2}_{\Omega^c\cup K}(\R^{N+1}_+;t^{1-2s})$. 
Then we can take $\phi=V_{\Omega,K}$ as a test function in \eqref{eq:V_K_weak}, so that
\[
\mathop{\rm Cap}\nolimits^s_\Omega(K)=q(V_{\Omega,K},V_{\Omega,K})=0.
\]

Now suppose that $\mathop{\rm Cap}\nolimits^s_\Omega(K)=0$.  We have
to show
$\mathcal{D}^{1,2}_{\Omega^c}(\R^{N+1}_+;t^{1-2s})\subset
\mathcal{D}^{1,2}_{\Omega^c\cup K}(\R^{N+1}_+;t^{1-2s})$,
the other inclusion being evident.  
 To this aim, let
$u\in \mathcal{D}^{1,2}_{\Omega^c}(\R^{N+1}_+;t^{1-2s})$.
By the assumption that
    $\mathop{\rm Cap}\nolimits^s_\Omega(K)=0$, for any $n\in\N$ there
  exists 
  $\eta_n\in C^\infty_c(\R^{N+1}_+\cup\Omega)$ such
  that $\eta_n \equiv 1$ in a neighborhood of $K$ and
\[
 \int_{\R^{N+1}_+} t^{1-2s} |\nabla \eta_n|^2 \,dx\,dt< \frac1n. 
\]
On the other hand, by density of
  $C^\infty_c(\R^{N+1}_+\cup \Omega)$ in
  $\mathcal{D}^{1,2}_{\Omega^c}(\R^{N+1}_+;t^{1-2s})$, for any
$\eps>0$ there 
exists $u_\eps\in C^\infty_c(\R^{N+1}_+\cup \Omega)$ such that
\[
 \|u_\eps-u\|^2_{\mathcal{D}^{1,2}(\R^{N+1}_+;t^{1-2s})}<\eps.
\]
In this way, the function $u_\eps(1-\eta_n) \in 
C^\infty_c(\R^{N+1}_+\cup (\Omega\setminus K))$; 
we estimate
\begin{align*}
  \int_{\R^{N+1}_+}& t^{1-2s} |\nabla (u_\eps(1-\eta_n) - u)|^2\,dx\,dt 
  = \int_{\R^{N+1}_+} t^{1-2s} |\nabla u_\eps - \nabla u - \nabla (\eta_n u_\eps)|^2\,dx\,dt \\
  &\leq 2  \int_{\R^{N+1}_+} t^{1-2s} |\nabla u_\eps - \nabla u|^2
    \,dx\,dt+
 2 \int_{\R^{N+1}_+} t^{1-2s} |\nabla (\eta_n u_\eps)|^2\,dx\,dt \\
&\leq 2\eps + 4 \int_{\R^{N+1}_+} t^{1-2s}|u_\eps|^2 |\nabla \eta_n
  |^2\,dx\,dt
 + 4 \int_{\R^{N+1}_+} t^{1-2s} |\eta_n|^2 |\nabla u_\eps|^2\,dx\,dt\\
&\leq 2\eps + \frac4n \sup|u_\eps|^2 + 4 \left(\sup |\nabla u_\eps|^2
  \right)
\int_{\mathop{\rm supp}
u_\eps} t^{1-2s} |\eta_n|^2\,dx\,dt \\
&\leq 2\eps + \frac4n \sup|u_\eps|^2 + \frac4n \left( \frac{2}{N-2s}
  \right)^2
 \sup |\nabla u_\eps|^2 \sup_{z\in\mathop{\rm supp} u_\eps}|z|^2,
\end{align*}
where the last relation relies on \eqref{eq:8}.

 This proves that $u$
  can be approximated in $\mathcal{D}^{1,2}(\R^{N+1}_+;t^{1-2s})$ with
  $C^\infty_c(\R^{N+1}_+\cup (\Omega\setminus K))$-functions, so that
  $u\in \mathcal{D}^{1,2}_{\Omega^c\cup K}(\R^{N+1}_+;t^{1-2s})$. 
\end{proof}

As a direct consequence of Proposition \ref{prop:cap0},
  we obtain that the removal of a zero fractional $s$-capacity set
  leaves the family of eigenvalues of $(-\Delta)^s$ unchanged.

\begin{corollary}
  Let $\Omega\subset\R^N$ be a bounded open set, $K\subset\Omega$
  compact and $s\in (0,\min\{1,N/2\})$.  It holds
  $\lambda_k^s(\Omega)=\lambda_k^s(\Omega\setminus K)$ for every
  $k\in \N_*$ if and only if
  $\mathop{\rm Cap}\nolimits^s_\Omega(K)=0$.
\end{corollary}
\begin{proof}
The result follows from Proposition \ref{prop:cap0} combined with
\eqref{eq:minimax} and the Spectral Theorem.
\end{proof}

\begin{remark}\label{rem:heinonen}
In the case $\Omega=\R^N$ and $K\subset\R^N$ compact, it holds
\[
2 \mathop{\rm Cap}\nolimits^s_{\R^N}(K)=\mathop{\rm Cap}\nolimits_{2,|t|^{1-2s}}(K,\R^{N+1}),
\]
where the right hand side of the above expression is the
$(2,|t|^{1-2s})$-capacity of the condenser $(K,\R^{N+1})$, as
introduced in \cite[Chapter 2]{heinonen2012nonlinear}. To see this, it
suffices to consider the function $V_K:=V_{\R^N,K}$ that achieves
$\mathop{\rm Cap}\nolimits^s_{\R^N}(K)$ and its even extension
\[
\tilde V_K(x,t)=\begin{cases}
V_K(x,t),&\text{if }t\geq0, \\
V_K(x,-t),&\text{if }t<0,
\end{cases}
\]
and to notice that
\[
\mathop{\rm Cap}\nolimits^s_{\R^N}(K)
=\frac{1}{2} \int_{\R^{N+1}} |t|^{1-2s} |\nabla\tilde V_K|^2 \,dx\,dt
=\frac{1}{2} \mathop{\rm Cap}\nolimits_{2,|t|^{1-2s}}(K,\R^{N+1}).
\]
We remark that $|t|^{1-2s}$ is a 2-admissible weight (according to the definition given in \cite[Chapter 2]{heinonen2012nonlinear}), since $|t|^{1-2s}$ belongs to the Muckenhoupt class $A_2$.
\end{remark}

Concerning the $s$-fractional $u$-capacity of $K$ in $\Omega$
introduced in Definition \ref{def:U-capacity} and characterized
equivalently in \eqref{eq:16}, we have that, as it happens for
$\mathop{\rm Cap}\nolimits^s_{\Omega}(K)$, the
infimum in \eqref{eq:16}  is achieved by a function
$V_{\Omega,K,u}\in \mathcal{D}^{1,2}_{\Omega^c}(\R^{N+1}_+;t^{1-2s})$
and the infimum in \eqref{eq:U_capacity_def} by $\mathop{\rm Tr}V_{\Omega,K,u}$, so
that
\begin{equation}\label{eq:V_K_U_def}
\mathop{\rm Cap}\nolimits^s_\Omega(K,u)= \frac1{\kappa_s}q(
V_{\Omega,K,u}, V_{\Omega,K,u})
=\|\mathop{\rm Tr}V_{\Omega,K,u}\|_{\mathcal D^{s,2}(\Omega)}^2,
\end{equation}
and $V_{\Omega,K,u}$ is the unique weak solution of
\begin{equation}\label{eq:V_K_U}
\begin{cases}
-\text{div}(t^{1-2s} \nabla V_{\Omega,K,u})=0, \quad&\text{in } \R^{N+1}_+, \\
\lim_{t\to0^+}\left(-t^{1-2s} \partial_t V_{\Omega,K,u} \right)=0,
&\text{in }(\Omega\setminus K)\times\{0\},\\
 V_{\Omega,K,u}=0, &\text{in } (\R^N\setminus\Omega)\times\{0\}, \\
V_{\Omega,K,u}=u, &\text{in } K\times\{0\},
\end{cases}
\end{equation}
in the sense that $V_{\Omega,K,u}\in
\mathcal{D}^{1,2}_{\Omega^c}(\R^{N+1}_+;t^{1-2s})$, $V_{\Omega,K,u} -U \in
\mathcal{D}^{1,2}_{\Omega^c\cup K}(\R^{N+1}_+;t^{1-2s})$
for some function $U\in  \mathcal{D}^{1,2}_{\Omega^c}(\R^{N+1}_+;t^{1-2s})$ 
such that $\mathop{\rm Tr}U=u$, 
and
\begin{equation}\label{eq:V_K_U_weak}
q(V_{\Omega,K,u},\phi )=0 \quad\text{for all } \phi\in \mathcal{D}^{1,2}_{\Omega^c\cup K}(\R^{N+1}_+;t^{1-2s}).
\end{equation}

\section{Continuity of the eigenvalue variation}\label{sec:continuity}

\begin{proof}[Proof of Theorem \ref{thm:continuity}]
 For every $j\in\{1,2,\dots,k\}$, let $\lambda_j^s(\Omega)$ and $U_j\in
 \mathcal{D}^{1,2}_{\Omega^c}(\R^{N+1}_+;t^{1-2s})$ solve \eqref{eq:realization_s} and \eqref{eq:normalization}.
Moreover we can choose the eigenfunctions $U_j$ in
  such a way that 
\begin{equation}\label{eq:orth}
\int_{\Omega}\mathop{\rm Tr}U_j(x) \mathop{\rm Tr}U_\ell(x)\,dx=0 \quad\text{for }j\neq\ell.
\end{equation}
Let us denote $u_j=\mathop{\rm Tr}U_j$ for all $j$. Let 
\[
E=\mathop{\rm span}\{\Phi_j:j=1,2,\dots,k\}
\ \subset  \mathcal{D}^{1,2}_{\Omega^c\cup K}(\R^{N+1}_+;t^{1-2s})
\]
where $\Phi_j=U_j(1-V_{\Omega,K})$ and $V_{\Omega,K}$ is the capacitary potential of $K$ satisfying \eqref{eq:V_K_weak}--\eqref{eq:V_K}.
We denote $\varphi_j=\mathop{\rm Tr}\Phi_j$ for all $j$ and $v_{\Omega,K}=\mathop{\rm Tr}V_{\Omega,K}$.
We observe that, in view of \eqref{eq:normalization}, \eqref{eq:orth}, \eqref{eq:hardy_bdd} and Lemma \ref{l:bk}, we have, for all $j,\ell\in \{1,\ldots,k\}$, 
\begin{align}\label{eq:5}
  \bigg|\int_\Omega \varphi_j(x) \varphi_\ell(x)\,dx-
  \delta_{j\ell}\bigg|&=\bigg|-2 \int_\Omega u_j(x) u_\ell(x)v_{\Omega,K}(x)\,dx+
\int_\Omega u_j(x) u_\ell(x)v_{\Omega,K}^2(x)\,dx \bigg|\\
\notag&\leq \left(\max_{1\leq j\leq k}\|u_j\|_{L^\infty(\Omega)}\right)^{\!2}\left(2\int_\Omega|v_{\Omega,K}(x)|\,dx
+\int_\Omega v_{\Omega,K}^2(x)\,dx\right)\\
\notag&\leq C \left((\mathop{\rm Cap}\nolimits^s_\Omega(K))^{1/2}+\mathop{\rm Cap}\nolimits^s_\Omega(K)\right)
\end{align}
for some constant $C>0$ independent of $K$.
On the other hand
\begin{multline*}
q(\Phi_j,\Phi_\ell) =
\int_{\R^{N+1}_+} t^{1-2s} (1-V_{\Omega,K})^2\nabla U_j\cdot \nabla U_\ell 
  \,dx\,dt +
\int_{\R^{N+1}_+} t^{1-2s} U_jU_\ell|\nabla V_{\Omega,K}|^2\,dx\,dt\\
- \int_{\R^{N+1}_+} t^{1-2s} U_j(1-V_{\Omega,K})\nabla U_\ell\cdot \nabla V_{\Omega,K}
  \,dx\,dt  -
\int_{\R^{N+1}_+} t^{1-2s} U_\ell(1-V_{\Omega,K})\nabla U_j\cdot \nabla V_{\Omega,K}
  \,dx\,dt  .
\end{multline*}
Choosing $\phi=U_\ell(1-V_{\Omega,K})^2$ in \eqref{eq:realization_s_w} we obtain that 
\begin{align*}
  \int_{\R^{N+1}_+} t^{1-2s}(1-V_{\Omega,K})^2\nabla U_j\cdot \nabla U_\ell
  \,dx\,dt&=2
  \int_{\R^{N+1}_+} t^{1-2s}(1-V_{\Omega,K})U_\ell\nabla U_j\cdot \nabla V_{\Omega,K}
  \,dx\,dt\\
&\quad+\kappa_s\lambda_j^s(\Omega)\int_{\Omega}\varphi_j(x) \varphi_\ell(x)\,dx,
\end{align*}
hence, thanks to Lemma \ref{l:bk} and \eqref{eq:5}, for every $j,\ell\in\{1,2,\dots,k\}$,
\begin{align}\label{eq:9}
\bigg|&\int_{\R^{N+1}_+} t^{1-2s} \nabla \Phi_j\cdot \nabla \Phi_\ell
  \,dx\,dt-\kappa_s\lambda_j^s(\Omega)\delta_{j\ell}\bigg|\\
\notag&\quad=\bigg|
  \int_{\R^{N+1}_+} t^{1-2s}(1-V_{\Omega,K})U_\ell\nabla U_j\cdot \nabla V_{\Omega,K}
  \,dx\,dt+\kappa_s\lambda_j^s(\Omega)\left(\int_{\Omega}\varphi_j(x) 
\varphi_\ell(x)\,dx-\delta_{j\ell}\right)\\
\notag&\quad\qquad+
\int_{\R^{N+1}_+} t^{1-2s} U_jU_\ell|\nabla V_{\Omega,K}|^2\,dx\,dt-
\int_{\R^{N+1}_+} t^{1-2s} U_j(1-V_{\Omega,K})\nabla U_\ell\cdot \nabla V_{\Omega,K}
  \,dx\,dt  \bigg|\\
\notag&\quad\leq C \left((\mathop{\rm Cap}\nolimits^s_\Omega(K))^{1/2}+\mathop{\rm Cap}\nolimits^s_\Omega(K)\right)
\end{align}
for some constant $C>0$ independent of $K$. The above estimate implies
there exists $\delta>0$ independent of $K$  such 
that $\Phi_1,\Phi_2,\dots,\Phi_k$ are linearly independent 
provided $\mathop{\rm Cap}\nolimits^s_\Omega(K)<\delta$, so that $E$
is a $k$-dimensional subspace of $\mathcal{D}^{1,2}_{\Omega^c\cup K}(\R^{N+1}_+;t^{1-2s})$ for $\mathop{\rm Cap}\nolimits^s_\Omega(K)<\delta$.

From \eqref{eq:minimax}, the fact that $\lambda_i^s(\Omega)\leq
\lambda_k^s(\Omega)$ for all $i\leq k$, \eqref{eq:5} and \eqref{eq:9} we have that 
\begin{align*}
  \lambda_k^s(\Omega\setminus K)&\leq
                                \max_{\substack{(\alpha_1,\alpha_2,\dots,\alpha_k)\in
                                \R^k\\
\sum_{i=1}^k\alpha_i^2=1}}\mathcal
  R\left(\sum_{i=1}^k\alpha_i\Phi_i\right)\\
&=\max_{\substack{(\alpha_1,\alpha_2,\dots,\alpha_k)\in
                                \R^k\\
\sum_{i=1}^k\alpha_i^2=1}}\frac{\sum_{i,j=1}^k\alpha_i\alpha_j 
q(\Phi_i,\Phi_j)}{\kappa_s \sum_{i,j=1}^k\alpha_i\alpha_j \int_{\Omega} \varphi_i\varphi_j\,dx}\\
&=\max_{\substack{(\alpha_1,\alpha_2,\dots,\alpha_k)\in
                                \R^k\\
\sum_{i=1}^k\alpha_i^2=1}}\frac{(\sum_{i=1}^k\alpha_i^2\kappa_s\lambda_i^s(\Omega))+O\left((\mathop{\rm Cap}\nolimits^s_\Omega(K))^{1/2}\right)}
{\kappa_s\left(1+O\left((\mathop{\rm
  Cap}\nolimits^s_\Omega(K))^{1/2}\right)\right)}\\
&\leq\frac{\kappa_s\lambda_k^s(\Omega)+O\left((\mathop{\rm Cap}\nolimits^s_\Omega(K))^{1/2}\right)}
{\kappa_s\left(1+O\left((\mathop{\rm
  Cap}\nolimits^s_\Omega(K))^{1/2}\right)\right)}=\lambda_k^s(\Omega)+O\left((\mathop{\rm Cap}\nolimits^s_\Omega(K))^{1/2}\right)\\
\end{align*}
as $\mathop{\rm Cap}\nolimits^s_\Omega(K)\to 0$. The proof is thereby complete.
\end{proof}

\section{Asymptotic expansion of the eigenvalues under removal of small capacity sets}\label{sec:expansion}

The aim of this section is to prove Theorem \ref{thm:expansion}.  The
proof is inspired from that of \cite[Theorem 1.4]{AFHL}.  Let us start
with some  preliminary lemmas concerning the capacitary potential
$V_{\Omega,K,u}$ defined in \eqref{eq:V_K_U_def}--\eqref{eq:V_K_U}.

\begin{lemma}\label{lem:trace_o_capacity}
Let $\{K_\varepsilon\}_{\varepsilon>0}$ be a family of compact sets
contained in the open set $\Omega$ concentrating, in the sense of Definition \ref{def:concentrates}, to a compact set $K\subset\Omega$,
with $\mathop{\rm Cap}\nolimits^s_\Omega(K)=0$. For every 
$u\in \mathcal D^{s,2}(\Omega)$ it holds
\begin{equation}
\int_\Omega |\mathop{\rm Tr}V_{\Omega,K_\varepsilon,u}|^2\,dx 
=o(\mathop{\rm Cap}\nolimits^s_\Omega(K_\varepsilon,u))
\qquad\text{as } \varepsilon\to 0.
\end{equation}
\end{lemma}
\begin{proof}
Let $\mathcal{H}_\varepsilon=\mathcal{D}^{1,2}_{\Omega^c\cup K_\varepsilon}(\R^{N+1}_+;t^{1-2s})$.
Suppose by contradiction that there exist a sequence $\varepsilon_n\to0$, $\varepsilon_n>0$, and a constant $C>0$ such that
\[
\int_\Omega |\mathop{\rm Tr}V_{\Omega,K_{\varepsilon_n},u}|^2\,dx 
\geq C \|V_{\Omega,K_{\varepsilon_n},u}\|^2_{\mathcal{D}^{1,2}_{\Omega^c}(\R^{N+1}_+;t^{1-2s})}
\]
for every $n$. Letting
\[
W_n=\frac{V_{\Omega,K_{\varepsilon_n},u}}{\| \mathop{\rm Tr} V_{K_{\Omega,K_{\varepsilon_n},u}}\|_{L^2(\Omega)}},
\]
we have
\[
\|\mathop{\rm Tr} W_n\|_{L^2(\Omega)}=1 \qquad\text{and}\qquad
\|W_n\|^2_{\mathcal{D}^{1,2}_{\Omega^c}(\R^{N+1}_+;t^{1-2s})} \leq C^{-1}
\]
for every $n$. 
By weak compactness of the unit ball of
$\mathcal{D}^{1,2}_{\Omega^c}(\R^{N+1}_+;t^{1-2s})$ and by compactness
of the trace operator $\mathop{\rm
  Tr}:\mathcal{D}^{1,2}_{\Omega^c}(\R^{N+1}_+;t^{1-2s}) \to
L^2(\Omega)$ (which follows easily by combining the continuity of
  the trace map
  $\mathop{\rm Tr}: \mathcal{D}^{1,2}_{\Omega^c}(\R^{N+1}_+;t^{1-2s})
  \to \mathcal{D}^{s,2}(\Omega)$ and part (iii) of Theorem
  \ref{thm:immersions}), there exist a subsequence $(n_k)_{k\geq1}$ and $W\in \mathcal{D}^{1,2}_{\Omega^c}(\R^{N+1}_+;t^{1-2s})$ such that
\begin{equation}\label{eq:W}
W_{n_k} \rightharpoonup W \quad \text{in }\mathcal{D}^{1,2}_{\Omega^c}(\R^{N+1}_+;t^{1-2s}) \text{ as }k\to+\infty 
\qquad\text{and}\qquad
\|\mathop{\rm Tr} W\|_{L^2(\Omega)}=1.
\end{equation}
Moreover, from \eqref{eq:V_K_U_weak} we deduce that
\[
\int_{\R^{N+1}_+} t^{1-2s} \nabla W_{n_k}\cdot\nabla\phi\,dx\,dt=0
\qquad\text{for every } \phi \in \mathcal{H}_{\varepsilon_{n_k}}.
\]
For every $\phi\in C^\infty_c(\R^{N+1}_+\cup (\Omega\setminus K))$, we have that $\phi\in \mathcal{H}_{\varepsilon}$ for $\varepsilon$ sufficiently small.
Therefore we can pass to the limit as $k\to+\infty$
above and obtain
\[
\int_{\R^{N+1}_+} t^{1-2s} \nabla W\cdot\nabla\phi\,dx\,dt=0
\qquad\text{for every } \phi \in C^\infty_c(\R^{N+1}_+\cup (\Omega\setminus K)).
\]
By density, the latter holds for every $\phi \in \mathcal{D}^{1,2}_{\Omega^c\cup K}(\R^{N+1}_+;t^{1-2s})$. Now, the assumption $\mathop{\rm Cap}\nolimits^s_\Omega(K)=0$ allows to deduce, through Proposition \ref{prop:cap0},
\[
\int_{\R^{N+1}_+} t^{1-2s} \nabla W\cdot\nabla\phi\,dx\,dt=0
\qquad\text{for every } \phi \in \mathcal{D}^{1,2}_{\Omega^c}(\R^{N+1}_+;t^{1-2s}).
\]
Hence we can replace $\phi=W$ in the previous identity thus obtaining that
  $\|W\|^2_{\mathcal{D}^{1,2}(\R^{N+1}_+;t^{1-2s})}=0$ and hence
$W\equiv0$ in $\overline{\R^{N+1}_+}$, thus contradicting \eqref{eq:W}.
\end{proof}

\begin{lemma}\label{lem:cap_cont}
Let $\{K_\varepsilon\}_{\varepsilon>0}$ be a family of compact sets
contained in the open set $\Omega$ concentrating, in the sense of Definition \ref{def:concentrates}, to a compact set
$K\subset\Omega$, with $\mathop{\rm Cap}\nolimits^s_\Omega(K)=0$. For
every $u\in \mathcal{D}^{s,2}(\Omega)$ it holds
\[
\lim_{\varepsilon\to0^+} \mathop{\rm Cap}\nolimits^s_\Omega(K_\varepsilon,u)
=\mathop{\rm Cap}\nolimits^s_\Omega(K,u)=0
\qquad\text{and}\qquad
V_{\Omega,K_\varepsilon,u}\to V_{\Omega,K,u}\equiv 0
\]
strongly in $\mathcal{D}^{1,2}(\R^{N+1}_+;t^{1-2s})$ as $\varepsilon\to0^+$.
\end{lemma}
\begin{proof}
Let  $U\in  \mathcal{D}^{1,2}_{\Omega^c}(\R^{N+1}_+;t^{1-2s})$ be
such that $\mathop{\rm Tr}U=u$
and let $V_{\Omega,K_\varepsilon,u}\in \mathcal{D}^{1,2}_{\Omega^c}(\R^{N+1}_+;t^{1-2s})$ achieve $\mathop{\rm Cap}\nolimits^s_\Omega(K_\varepsilon,u)$. Then, by \eqref{eq:V_K_U_weak}, $V_{\Omega,K_\varepsilon,u} -U \in
\mathcal{D}^{1,2}_{\Omega^c\cup K_\varepsilon}(\R^{N+1}_+;t^{1-2s})$ and
\begin{equation}\label{eq:V_K_U_weak_eps}
q(V_{\Omega,K_\varepsilon,u},\phi )=0 \quad\text{for all } \phi\in \mathcal{D}^{1,2}_{\Omega^c\cup K_\varepsilon}(\R^{N+1}_+;t^{1-2s}).
\end{equation}
 As $V_{\Omega,K_\varepsilon,u}$ achieves \eqref{eq:16}, we have
\[
\|V_{\Omega,K_\varepsilon,u}\|^2_{\mathcal{D}^{1,2}(\R^{N+1}_+;t^{1-2s})}
\leq 
\|U\|^2_{\mathcal{D}^{1,2}(\R^{N+1}_+;t^{1-2s})},
\]
so that $\{V_{\Omega,K_\varepsilon,u}\}_{\varepsilon>0}$ is bounded in
$\mathcal{D}^{1,2}_{\Omega^c}(\R^{N+1}_+;t^{1-2s})$. Hence there exist
a sequence $\varepsilon_n\to0^+$ and
$V\in \mathcal{D}^{1,2}_{\Omega^c}(\R^{N+1}_+;t^{1-2s})$ such that
$V_{\Omega,K_{\varepsilon_n},u} \rightharpoonup V$ weakly in
$\mathcal{D}^{1,2}_{\Omega^c}(\R^{N+1}_+;t^{1-2s})$. Let us show that
$V=V_{\Omega,K,u}$. On the one hand,
$V-U \in \mathcal{D}^{1,2}_{\Omega^c\cup K}(\R^{N+1}_+;t^{1-2s})$
thanks to Proposition \ref{prop:cap0} and the assumption
$\mathop{\rm Cap}\nolimits^s_\Omega(K)=0$. On the other hand, passing
to the limit in \eqref{eq:V_K_U_weak_eps} we obtain that $q(V,\phi)=0$ for
every $\phi\in C^\infty_c(\R^{N+1}_+\cup (\Omega\setminus K))$ and so,
by density, for every
$\phi \in \mathcal{D}^{1,2}_{\Omega^c\cup
  K}(\R^{N+1}_+;t^{1-2s})=\mathcal{D}^{1,2}_{\Omega^c}(\R^{N+1}_+;t^{1-2s})$.
Therefore $V=V_{\Omega,K,u}\equiv 0$. In order to prove that the convergence
is strong, take $\phi=V_{\Omega,K_{\varepsilon_n},u}-U$ in
\eqref{eq:V_K_U_weak_eps} and pass to the limit to obtain
\[
\lim_{n\to+\infty}\mathop{\rm Cap}\nolimits^s_\Omega(K_{\varepsilon_n},u)
=\lim_{n\to+\infty} q(V_{\Omega,K_{\varepsilon_n},u},V_{\Omega,K_{\varepsilon_n},u})
=\lim_{n\to+\infty} q(V_{\Omega,K_{\varepsilon_n},u},U)=q(V,U)=0.
\]
We conclude that
$\mathop{\rm Cap}\nolimits^s_\Omega(K_{\varepsilon_n},u)\to 0$ and
that $V_{\Omega,K_{\varepsilon_n},u}\to0$ strongly in
$\mathcal{D}^{1,2}_{\Omega^c}(\R^{N+1}_+;t^{1-2s})$. Since these
limits do not depend on the sequence $\varepsilon_n\to0$, we reach the conclusion.
\end{proof}

Let us introduce the operator $A:
\mathcal{D}^{1,2}_{\Omega^c}(\R^{N+1}_+;t^{1-2s}) 
\to \mathcal{D}^{1,2}_{\Omega^c}(\R^{N+1}_+;t^{1-2s})$ defined by
\begin{equation}\label{eq:A_def}
q(A(U),V)=\int_\Omega \mathop{\rm Tr}U\mathop{\rm Tr V} \,dx
\end{equation}
for every $U,V\in
\mathcal{D}^{1,2}_{\Omega^c}(\R^{N+1}_+;t^{1-2s})$.
It is straightforward to see that $A$ is symmetric, nonnegative, and
compact.  Letting, for $j\in\N_*$,
\begin{equation}\label{eq:mu_def}
\mu_j=\frac{1}{\kappa_s \lambda_j^s(\Omega)},
\end{equation}
the spectrum of $A$ is $\{0\}\cup\{\mu_j:j\in\N_*\}$;
furthermore, since $\textrm{dim ker}A=+\infty$, $0$ has
  infinite multiplicity as an eigenvalue of $A$, whereas the non-zero
  eigenvalues of $A$ have finite multiplicity.

\begin{proof}[Proof of Theorem \ref{thm:expansion}]
Let $U_j=\mathcal H(u_j)$, so that $U_j$  satisfies
\eqref{eq:realization_s} and \eqref{eq:normalization}.
To simplify the notation, in the rest of the proof we write
$V_\varepsilon=V_{\Omega,K_\varepsilon,u_j}$
 and $\mathcal{H}_\varepsilon=\mathcal{D}^{1,2}_{\Omega^c\cup K_\varepsilon}(\R^{N+1}_+;t^{1-2s})$. 
We divide the proof into three steps.

\medskip

\noindent {\bf Step 1}. We claim that
\begin{equation}\label{eq:claim1_expansion}
\lambda_j^s(\Omega\setminus K_\varepsilon)-\lambda_j^s(\Omega)=
o\left(\sqrt{\mathop{\rm Cap}\nolimits^s_\Omega(K_\varepsilon,u_j)}\right)
\qquad\text{as }\varepsilon\to0^+.
\end{equation}
Let 
\begin{equation}\label{eq:psi_eps}
\psi_\varepsilon=U_j-V_\varepsilon \in \mathcal{H}_\varepsilon, 
\end{equation}
so that $\psi_\varepsilon$ is the orthogonal projection of $U_j$ on $\mathcal{H}_\varepsilon$ in the space $\mathcal{D}^{1,2}_{\Omega^c}(\R^{N+1}_+;t^{1-2s})$ endowed with the scalar product $q$, that is
\[
q(\psi_\varepsilon-U_j,\phi)=0 \qquad\text{for every }\phi \in \mathcal{H}_\varepsilon.
\]
For every $\phi \in \mathcal{H}_\varepsilon$ we have, using \eqref{eq:realization_s_w},
\begin{align*}
  q(\psi_\varepsilon,\phi)-\kappa_s\lambda_j^s(\Omega)
  \int_\Omega \mathop{\rm Tr}\psi_\varepsilon\mathop{\rm Tr}\phi \,dx
  &=q(U_j,\phi)-\kappa_s\lambda_j^s(\Omega)
    \int_\Omega \mathop{\rm Tr}\psi_\varepsilon\mathop{\rm Tr}\phi \,dx \\
  &=\kappa_s\lambda_j^s(\Omega)\int_\Omega \mathop{\rm Tr}V_\varepsilon\mathop{\rm Tr}\phi \,dx,
\end{align*}
so that
\begin{equation}\label{eq:Tr_psi_eps}
\int_\Omega \mathop{\rm Tr}\psi_\varepsilon\mathop{\rm Tr}\phi \,dx
=\frac{1}{\kappa_s\lambda_j^s(\Omega)} q(\psi_\varepsilon,\phi)-\int_\Omega \mathop{\rm Tr}V_\varepsilon\mathop{\rm Tr}\phi \,dx \qquad \text{for ever }\phi\in \mathcal{H}_\varepsilon.
\end{equation}
Let $A_\varepsilon:\mathcal{H}_\varepsilon\to\mathcal{H}_\varepsilon$ be defined by
\begin{equation}\label{eq:A_eps_def}
q(A_\varepsilon(U),V)=\int_\Omega \mathop{\rm Tr}U\mathop{\rm Tr V} \,dx
\qquad \text{for every }U,V\in \mathcal{H}_\varepsilon.
\end{equation}
Recalling the definition of $\mu_j$ in \eqref{eq:mu_def}, the spectral
theorem (see for instance \cite[Proposition 8.20]{HelfferBook}) provides
\begin{equation}\label{eq:spectral_cons}
\text{dist}(\mu_j,\sigma(A_\varepsilon)) \leq \frac{\|A_\varepsilon\psi_\varepsilon-\mu_j\psi_\varepsilon\|_{\mathcal{H}_\varepsilon}}{\|\psi_\varepsilon\|_{\mathcal{H}_\varepsilon}},
\end{equation}
where $\sigma(A_\varepsilon)$ is the spectrum of $A_\varepsilon$.

Taking into account  Lemma \ref{lem:cap_cont}, we have
  that  
\[
| q(U_j,V_\varepsilon)|\leq
  \sqrt{q(U_j,U_j)}\sqrt{q(V_\varepsilon,V_\varepsilon)}=
\sqrt{\lambda_j^s(\Omega)\kappa_s}\sqrt{\mathop{\rm Cap}\nolimits^s_\Omega(K_\varepsilon,u_j)}
=o(1)
\]
as $\varepsilon\to 0$, then the denominator in the right hand side of
\eqref{eq:spectral_cons} is easily estimated as follows
\begin{align}\label{eq:donom_spectral}
  \|\psi_\varepsilon\|^2_{\mathcal{H}_\varepsilon}
  &=q(U_j-V_\varepsilon,U_j-V_\varepsilon) =
  q(U_j,U_j)+\mathop{\rm Cap}\nolimits^s_\Omega(K_\varepsilon,u_j) -2
  q(U_j,V_\varepsilon)\\
&\notag=\lambda_j^s(\Omega)\kappa_s+o(1)\quad\text{as }
    \varepsilon\to0^+.
\end{align}
In order to estimate the numerator in the right hand side of
\eqref{eq:spectral_cons}, let
$Z_\varepsilon=A_\varepsilon\psi_\varepsilon-\mu_j\psi_\varepsilon
\in\mathcal{H}_\varepsilon $. Using
\eqref{eq:A_eps_def} and \eqref{eq:Tr_psi_eps}, we have
\[
q(Z_\varepsilon,\phi)+\mu_j q(\psi_\varepsilon,\phi)=q(A_\varepsilon\psi_\varepsilon,\phi)
=\int_\Omega \mathop{\rm Tr}\psi_\varepsilon \mathop{\rm Tr}\phi \,dx
=\mu_j q(\psi_\varepsilon,\phi)-\int_\Omega \mathop{\rm Tr}V_\varepsilon \mathop{\rm Tr}\phi \,dx,
\]
for every $\phi\in \mathcal{H}_\varepsilon$. Choosing
$\phi=Z_\varepsilon \in\mathcal{H}_\varepsilon$ in the previous
expression and using Theorem \ref{thm:immersions} (i) and
\eqref{eq:4}, we obtain
\begin{align}\label{eq:num_spectral}
\|Z_\varepsilon\|^2_{\mathcal{H}_\varepsilon} 
&=-\int_\Omega \mathop{\rm Tr}V_\varepsilon \mathop{\rm Tr}Z_\varepsilon \,dx\leq  
\|\mathop{\rm Tr}
V_\varepsilon\|_{L^2(\Omega)} |\Omega|^{\frac sN}\left(
\int_\Omega  |\mathop{\rm
  Tr}Z_\varepsilon|^{2^*(s)}\right)^{\!\!\frac1{2^*(s)}}\\
&\notag \leq \|\mathop{\rm Tr}
V_\varepsilon\|_{L^2(\Omega)}|\Omega|^{\frac sN}S_{N,s}^{-1}\|\mathop{\rm
  Tr}Z_\varepsilon\|_{\mathcal D^{s,2}(\R^N)}
\leq \|\mathop{\rm Tr}
V_\varepsilon\|_{L^2(\Omega)}|\Omega|^{\frac sN}S_{N,s}^{-1}\kappa_s^{-1/2}\|Z_\varepsilon\|_{\mathcal H_\varepsilon}.
\end{align}
Replacing \eqref{eq:donom_spectral} and \eqref{eq:num_spectral} into \eqref{eq:spectral_cons}, we find that there exists a constant $C$ independent of $\varepsilon$ such that 
\begin{equation}\label{eq:spectral_cons2}
\text{dist}(\mu_j,\sigma(A_\varepsilon)) \leq C 
\|\mathop{\rm Tr} V_\varepsilon\|_{L^2(\Omega)}.
\end{equation}
Now, the assumption that $\lambda_j^s(\Omega)$ is simple and the continuity proved in Theorem \ref{thm:continuity} imply that
\[
\lambda_{j,\varepsilon}:=\lambda_j^s(\Omega\setminus K_\varepsilon)
\qquad \text{is simple for }\varepsilon>0\text{ small enough}.
\]
Denoting as
\begin{equation}\label{eq:mu_eps}
\mu_{j,\varepsilon}=1/(\kappa_s \lambda_{j,\varepsilon})
\end{equation}
the $j$-th eigenvalue of $A_\eps$, by the simplicity of
  $\mu_j$ as an eigenvalue of the operator $A$ introduced in \eqref{eq:A_def}, and by Theorem \ref{thm:continuity} we
  have that 
\[
\text{dist}(\mu_j,\sigma(A_\varepsilon))=|\mu_j-\mu_{j,\varepsilon}|
\quad \text{for }\varepsilon>0\text{ small enough}.
\]
Then relation \eqref{eq:spectral_cons2} provides, for
$\eps$ small enough,
\[
|\lambda_j^s(\Omega)-\lambda_{j,\varepsilon}|
=\kappa_s\lambda_j^s(\Omega)\lambda_{j,\varepsilon} |\mu_j-\mu_{j,\varepsilon}|
\leq C \kappa_s\lambda_j^s(\Omega)\lambda_{j,\varepsilon} \|\mathop{\rm Tr} V_\varepsilon\|_{L^2(\Omega)}.
\]
As $C$ is independent of $\varepsilon$ and
  $\lim_{\eps\to0^+}\lambda_{j,\varepsilon}=\lambda_j^s(\Omega)$,
Lemma \ref{lem:trace_o_capacity} provides the claim.

\medskip

\noindent {\bf Step 2}. We claim that
\begin{equation}\label{eq:claim2_expansion}
\|\psi_\varepsilon-\Pi_\varepsilon\psi_\varepsilon\|_{\mathcal{H}_\varepsilon} 
=o\left(\sqrt{\mathop{\rm Cap}\nolimits^s_\Omega(K_\varepsilon,u_j)}\right)
\qquad\text{as }\varepsilon\to0^+,
\end{equation}
where $\Pi_\varepsilon : \mathcal{D}^{1,2}_{\Omega^c}(\R^{N+1}_+;t^{1-2s})\to \mathcal{H}_\varepsilon$ is defined as
\[
\Pi_\varepsilon W=\left(\int_\Omega \mathop{\rm Tr}W\mathop{\rm Tr}
  U_{j,\varepsilon}\,dx\right)
\, U_{j,\varepsilon}
\quad \text{for any }W\in \mathcal{D}^{1,2}_{\Omega^c}(\R^{N+1}_+;t^{1-2s})\]
and $U_{j,\varepsilon}$ is a normalized eigenfunction associated to
$\lambda_{j,\varepsilon}$, i.e. 
\begin{equation}\label{eq:18}
\begin{cases}
U_{j,\eps}\in\mathcal H_\eps,\\
{\displaystyle{ q(U_{j,\eps},\phi)
=\lambda_{j,\eps}\kappa_s\int_\Omega \mathop{\rm
  Tr}U_{j,\eps}\mathop{\rm Tr}\phi\,dx}}\quad\text{for every
}\phi\in\mathcal H_\eps,\\
\int_{\Omega}|\mathop{\rm Tr}U_{j,\eps}(x)|^2\,dx =1.
\end{cases}
\end{equation}
Let $\tilde{U}_\varepsilon=\psi_\varepsilon-\Pi_\varepsilon\psi_\varepsilon$ and notice that
\begin{equation}\label{eq:19}
\int_\Omega \mathop{\rm Tr}\tilde{U}_\varepsilon \mathop{\rm Tr} U_{j,\varepsilon} \,dx=0.
\end{equation}
Using the fact that $\Pi_\varepsilon\psi_\varepsilon$ is an eigenfunction associated to $\lambda_{j,\varepsilon}$ and relation \eqref{eq:Tr_psi_eps}, we see that the following holds for every $\phi\in \mathcal{H}_\varepsilon$
\begin{align}\label{eq:q_U_tilde}
  &q(\tilde{U}_\varepsilon,\phi)-\kappa_s\lambda_{j,\varepsilon}
    \int_\Omega \mathop{\rm Tr}\tilde{U}_\varepsilon \mathop{\rm Tr}\phi\,dx \\
  \notag &\quad=
           q(\psi_\varepsilon,\phi)-\kappa_s\lambda_{j,\varepsilon}
           \int_\Omega \mathop{\rm Tr}\psi_\varepsilon \mathop{\rm Tr}\phi\,dx
           -\left[q(\Pi_\varepsilon\psi_\varepsilon,\phi)-
           \kappa_s\lambda_{j,\varepsilon} 
           \int_\Omega \mathop{\rm Tr}(\Pi_\varepsilon\psi_\varepsilon)\mathop{\rm Tr}\phi\,dx \right]\\
  \notag&\quad=q(\psi_\varepsilon,\phi)-\kappa_s\lambda_j^s(\Omega)
          \int_\Omega \mathop{\rm Tr}\psi_\varepsilon \mathop{\rm Tr}\phi\,dx 
          +\kappa_s(\lambda_j^s(\Omega)-\lambda_{j,\varepsilon}) 
          \int_\Omega \mathop{\rm Tr}\psi_\varepsilon \mathop{\rm Tr}\phi\,dx \\
  \notag&\quad=\kappa_s\lambda_j^s(\Omega)
          \int_\Omega \mathop{\rm Tr}V_\varepsilon \mathop{\rm Tr}\phi\,dx 
          +\kappa_s(\lambda_j^s(\Omega)-
          \lambda_{j,\varepsilon}) \int_\Omega \mathop{\rm Tr}\psi_\varepsilon \mathop{\rm Tr}\phi\,dx .
\end{align}
Let $\xi_\varepsilon=A_\varepsilon(\tilde{U}_\varepsilon)-\mu_{j,\varepsilon}\tilde{U}_\varepsilon \in \mathcal{H}_\varepsilon$. 
We use the definition of $A_\varepsilon$ in \eqref{eq:A_eps_def},  
that of $\mu_{j,\varepsilon}$ in \eqref{eq:mu_eps} and
relation \eqref{eq:q_U_tilde} evaluated at $\phi=\xi_\varepsilon$
to compute
\begin{multline*}
  \|\xi_\varepsilon\|^2_{\mathcal{H}_\varepsilon} = q(A_\varepsilon
  (\tilde{U}_\varepsilon),\xi_\varepsilon)-\mu_{j,\varepsilon}
  q(\tilde{U}_\varepsilon,\xi_\varepsilon)
  =\int_\Omega \mathop{\rm Tr}\tilde{U}_\varepsilon \mathop{\rm Tr}\xi_\varepsilon \,dx \\
  -\left[ \int_\Omega \mathop{\rm Tr}\tilde{U}_\varepsilon \mathop{\rm
      Tr}\xi_\varepsilon \,dx
    +\frac{\lambda_j^s(\Omega)}{\lambda_{j,\varepsilon}}\int_\Omega
    \mathop{\rm Tr}V_\varepsilon \mathop{\rm Tr}\xi_\varepsilon \,dx
    +\frac{\lambda_j^s(\Omega)-\lambda_{j,\varepsilon}}{\lambda_{j,\varepsilon}}
 \int_\Omega \mathop{\rm Tr}\psi_\varepsilon \mathop{\rm Tr}\xi_\varepsilon \,dx \right]\\
  =-\frac{\lambda_j^s(\Omega)}{\lambda_{j,\varepsilon}}\int_\Omega
  \mathop{\rm Tr}V_\varepsilon \mathop{\rm Tr}\xi_\varepsilon \,dx
  -\frac{\lambda_j^s(\Omega)-\lambda_{j,\varepsilon}}{\lambda_{j,\varepsilon}}
  \int_\Omega \mathop{\rm Tr}\psi_\varepsilon \mathop{\rm
    Tr}\xi_\varepsilon \,dx ,
\end{multline*}
from which, taking into account \eqref{eq:donom_spectral} and \eqref{eq:hardy_bdd},
we deduce that
\[
\|\xi_\varepsilon\|_{\mathcal{H}_\varepsilon} \leq C\left( \|\mathop{\rm Tr} V_\varepsilon\|_{L^2(\Omega)}+|\lambda_j^s(\Omega)-\lambda_{j,\varepsilon}|\right),
\]
for a constant $C$ not depending on $\varepsilon$. Lemma \ref{lem:trace_o_capacity} and relation \eqref{eq:claim1_expansion} provide then
\begin{equation}\label{eq:tildeU}
\|A_\varepsilon(\tilde{U}_\varepsilon)-\mu_{j,\varepsilon}
\tilde{U}_\varepsilon\|_{\mathcal{H}_\varepsilon}=
\|\xi_\varepsilon\|_{\mathcal{H}_\varepsilon}=
o\left(\sqrt{\mathop{\rm Cap}\nolimits^s_\Omega(K_\varepsilon,u_j)}\right)
\qquad\text{as }\varepsilon\to0^+.
\end{equation}
Let
\[
\mathcal{K}_\varepsilon=\mathop{\rm Ker} \Pi_\varepsilon \vert_{\mathcal{H}_\varepsilon}
=\left\{W\in \mathcal{H}_\varepsilon : \, \int_\Omega \mathop{\rm Tr}W \mathop{\rm Tr} U_{j,\varepsilon}\,dx=0\right\}
\]
and note that $\tilde U_\eps\in \mathcal K_\eps$ thanks
  to \eqref{eq:19}.  Moreover, in view of \eqref{eq:A_eps_def} and \eqref{eq:18},
  $A_\eps(U)\in \mathcal{K}_\varepsilon$ for all
  $U\in \mathcal{K}_\varepsilon$,  hence, denoting as $\tilde{A}_\varepsilon$ 
 the restriction of $A_\varepsilon$ to $\mathcal{K}_\varepsilon$, we
have 
$\tilde{A}_\varepsilon :\mathcal{K}_\varepsilon\to
\mathcal{K}_\varepsilon$. As
$\sigma(\tilde{A}_\varepsilon)=\sigma(A_\varepsilon)\setminus\{\mu_{j,\varepsilon}\}$,
there exists $\delta>0$ independent of $\varepsilon$ such that
$\text{dist}(\mu_{j,\varepsilon},\sigma(\tilde{A}_\varepsilon))\geq
\delta$.
We use this inequality, the spectral theorem, and relation
\eqref{eq:tildeU} to obtain
\begin{align*}
\|\psi_\varepsilon-\Pi_\varepsilon\psi_\varepsilon\|_{\mathcal{H}_\varepsilon} =
\|\tilde{U}_\varepsilon\|_{\mathcal{H}_\varepsilon} &\leq 
\frac{1}{\delta} \text{dist}(\mu_{j,\varepsilon},\sigma(\tilde{A}_\varepsilon)) \|\tilde{U}_\varepsilon\|_{\mathcal{H}_\varepsilon}\\
&\leq \frac{1}{\delta} \| \tilde{A}_\varepsilon(\tilde{U}_\varepsilon)-\mu_{j,\varepsilon}\tilde{U}_\varepsilon\|_{\mathcal{H}_\varepsilon} 
=o\left(\sqrt{\mathop{\rm Cap}\nolimits^s_\Omega(K_\varepsilon,u_j)}\right)
\end{align*}
as $\eps\to0^+$, thus proving \eqref{eq:claim2_expansion}.

\medskip

\noindent {\bf Step 3}. From the definition of $\psi_\varepsilon$ \eqref{eq:psi_eps},
\eqref{eq:normalization}, Lemma \ref{lem:trace_o_capacity},
\eqref{eq:claim2_expansion} and \eqref{eq:hardy_bdd}, we have
\begin{align}\label{eq:tracePipsi}
\|\mathop{\rm Tr}(\Pi_\varepsilon\psi_\varepsilon)\|_{L^2(\Omega)}
&=\left(\int_\Omega \left|\mathop{\rm
  Tr}(\Pi_\varepsilon\psi_\varepsilon-\psi_\varepsilon)+u_j-
\mathop{\rm Tr}V_\eps\right|^2\,dx
\right)^{\!1/2}\\
\notag&=\left(1+ o\left(\sqrt{\mathop{\rm Cap}\nolimits^s_\Omega(K_\varepsilon,u_j)}\right)\right)^{\!1/2}\\
\notag&=1+o\left(\sqrt{\mathop{\rm Cap}\nolimits^s_\Omega(K_\varepsilon,u_j)}\right)
\qquad\text{as }\varepsilon\to0^+.
\end{align}
Let
\[
\Psi_\varepsilon=\frac{\Pi_\varepsilon\psi_\varepsilon}{\|\mathop{\rm Tr}(\Pi_\varepsilon\psi_\varepsilon)\|_{L^2(\Omega)}} \in \mathcal{H}_\varepsilon.
\]
Noticing that
\[
\Psi_\varepsilon-\psi_\varepsilon=
\frac{\Pi_\varepsilon\psi_\varepsilon-{\psi_\eps}+(1-\|\mathop{\rm Tr}(\Pi_\varepsilon\psi_\varepsilon)\|_{L^2(\Omega)})\psi_\varepsilon}{\|\mathop{\rm Tr}(\Pi_\varepsilon\psi_\varepsilon)\|_{L^2(\Omega)}}
\]
and using \eqref{eq:tracePipsi}, \eqref{eq:claim2_expansion} and \eqref{eq:hardy_bdd}, we deduce that
\begin{equation}\label{eq:trace_est1}
\|\mathop{\rm Tr}(\Psi_\varepsilon-\psi_\varepsilon)\|_{L^2(\Omega)}
=o\left(\sqrt{\mathop{\rm Cap}\nolimits^s_\Omega(K_\varepsilon,u_j)}\right)
\qquad\text{as }\varepsilon\to0^+.
\end{equation}
Similarly,
\begin{equation}\label{eq:trace_est2}
\|\mathop{\rm Tr}(\Psi_\varepsilon-U_j)\|_{L^2(\Omega)}
=\|\mathop{\rm Tr}(\Psi_\varepsilon-\psi_\varepsilon-V_\varepsilon)\|_{L^2(\Omega)}
=o\left(\sqrt{\mathop{\rm Cap}\nolimits^s_\Omega(K_\varepsilon,u_j)}\right)
\qquad\text{as }\varepsilon\to0^+.
\end{equation}
We also remark, using the equation satisfied by $V_\varepsilon$ (see \eqref{eq:V_K_U_weak}), the fact that $\psi_\varepsilon\in\mathcal{H}_\varepsilon$ and the equation satisfied by $U_j$, that
\begin{equation}\label{eq:cap_equiv}
  \mathop{\rm Cap}\nolimits^s_\Omega(K_\varepsilon,u_j)=
{\frac1{\kappa_s}}q(V_\varepsilon,V_\varepsilon)
  ={\frac1{\kappa_s}}q(V_\varepsilon,U_j-\psi_\varepsilon)=
{\frac1{\kappa_s}}q(V_\varepsilon,U_j)
  =\lambda_j^s(\Omega)\int_\Omega u_j \mathop{\rm Tr}V_\varepsilon \,dx.
\end{equation}
Noticing that $\Psi_\varepsilon$ is an eigenfunction associated to $\lambda_{j,\varepsilon}$, relation \eqref{eq:Tr_psi_eps} with $\phi=\Psi_\varepsilon$ provides
\[
(\lambda_{j,\varepsilon}-\lambda_j^s(\Omega))\int_\Omega \mathop{\rm Tr}\Psi_\varepsilon\mathop{\rm Tr}\psi_\varepsilon \,dx =
\lambda_j^s(\Omega)\int_\Omega \mathop{\rm Tr}\Psi_\varepsilon\mathop{\rm Tr}V_\varepsilon \,dx.
\]
Therefore, by \eqref{eq:trace_est2}, \eqref{eq:cap_equiv} and Lemma
\ref{lem:trace_o_capacity}, we have
\begin{align*}
(\lambda_{j,\varepsilon}-\lambda_j^s(\Omega))\int_\Omega \mathop{\rm Tr}\Psi_\varepsilon\mathop{\rm Tr}\psi_\varepsilon \,dx 
&=\lambda_j^s(\Omega) \int_\Omega u_j\mathop{\rm
  Tr}V_\varepsilon\,dx+
\lambda_j^s(\Omega) \int_\Omega \mathop{\rm Tr}(\Psi_\varepsilon- U_j) \mathop{\rm Tr}V_\varepsilon\,dx\\
&=\mathop{\rm Cap}\nolimits^s_\Omega(K_\varepsilon,u_j)
+o(\mathop{\rm Cap}\nolimits^s_\Omega(K_\varepsilon,u_j))
\end{align*}
as $\varepsilon\to0^+$. As, by \eqref{eq:trace_est1},
\[
\int_\Omega \mathop{\rm Tr}\Psi_\varepsilon\mathop{\rm Tr}\psi_\varepsilon \,dx
=\int_\Omega |\mathop{\rm Tr}\Psi_\varepsilon|^2 \,dx
+\int_\Omega \mathop{\rm Tr}\Psi_\varepsilon\mathop{\rm Tr}(\psi_\varepsilon-\Psi_\varepsilon) \,dx
=1+o\left(\sqrt{\mathop{\rm Cap}\nolimits^s_\Omega(K_\varepsilon,u_j)}\right),
\]
we have concluded the proof.
\end{proof}

\section{Asymptotics of capacities for scaling of a given set}\label{sec:scaling}

In this section we will assume that $0\in\Omega$.
In order to prove Theorem \ref{thm:cap_K_eps}, we first establish the following preliminary result.

\begin{lemma}\label{lem:cap_f_n}
  Let $K\subset\Omega$ be compact and 
 $\Omega'$ be an open set such that
  $K\subset\Omega'\Subset\Omega$. Let $f\in H^s_{\rm loc}(\Omega)$ and 
  $(f_n)_{n\geq1}\subset H^s_{\rm loc}(\Omega)$ be such that $f_n\to f$
  as $n\to+\infty$ in $H^s(\Omega')$. Then
\[
 V_{\Omega,K,f_n}\to V_{\Omega,K,f}\quad\text{in }\mathcal{D}^{1,2}(\R^{N+1};t^{1-2s})
\]
and
\[
\lim_{n\to+\infty} \mathop{\rm Cap}\nolimits^s_\Omega(K,f_n)=
\mathop{\rm Cap}\nolimits^s_\Omega(K,f).
\]
\end{lemma}
\begin{proof}
Let $\tilde\eta_K\in C^\infty(\R^{N+1}_+\cup\Omega')$
  be such that $\tilde\eta_K\equiv 1$ in a neighborhood of $K$.
Therefore $\tilde\eta_K f_n\to \tilde\eta_K f$ in $\mathcal
D^{s,2}(\Omega')$ and, consequently, $\mathcal H(\tilde\eta_K f_n)\to \mathcal
H(\tilde\eta_K f)$ in
$\mathcal{D}^{1,2}(\R^{N+1};t^{1-2s})$, where $\mathcal{H}$ is the extension operator introduced in \eqref{eq:13}. 

Furthermore both 
$V_{\Omega,K,f_n}-\mathcal H(\tilde\eta_K f_n)$
 and $V_{\Omega,K,f}-\mathcal H(\tilde\eta_K f)$ belong to $\mathcal{D}^{1,2}_{\Omega^c\cup K}(\R^{N+1};t^{1-2s})$. Hence
\[
q(V_{\Omega,K,f_n}-V_{\Omega,K,f},V_{\Omega,K,f_n}-\mathcal H(\tilde\eta_K f_n))=
q(V_{\Omega,K,f_n}-V_{\Omega,K,f},V_{\Omega,K,f}-\mathcal H(\tilde\eta_K f))=0,
\]
so that, using the H\"older inequality,
\begin{multline*}
\| V_{\Omega,K,f_n}-V_{\Omega,K,f}\|^2_{\mathcal{D}^{1,2}(\R^{N+1};t^{1-2s})}
=q(V_{\Omega,K,f_n}-V_{\Omega,K,f}, \mathcal H(\tilde\eta_K f_n)-\mathcal H(\tilde\eta_K f))\\
\leq \| V_{\Omega,K,f_n}-V_{\Omega,K,f}\|_{\mathcal{D}^{1,2}(\R^{N+1};t^{1-2s})}
\|\mathcal H(\tilde\eta_K f_n)-\mathcal H(\tilde\eta_K f)\|_{\mathcal{D}^{1,2}(\R^{N+1};t^{1-2s})}.
\end{multline*}
Then 
\[
\lim_{n\to+\infty} \|V_{\Omega,K,f_n}-V_{\Omega,K,f}\|_{\mathcal{D}^{1,2}(\R^{N+1};t^{1-2s})}=0,
\]
concluding the proof.
\end{proof}

\begin{proof}[Proof of Theorem \ref{thm:cap_K_eps}]
For every $\varepsilon>0$, let $V_{\Omega,K_\varepsilon,u_j}$ be the function that achieves $\mathop{\rm Cap}\nolimits^s_\Omega(K_\varepsilon,u_j)$ as in \eqref{eq:V_K_U_def} and let
\[
\tilde V_\varepsilon(z)=\varepsilon^{-\gamma_s}
V_{\Omega,K_\varepsilon,u_j} (\varepsilon z), \quad
 z\in\R^{N+1}_+.
\]
Let $U_j=\mathcal H(u_j) \in
 \mathcal{D}^{1,2}_{\Omega^c}(\R^{N+1}_+;t^{1-2s})$ be the extension
 of $u_j$ as in \eqref{eq:13} and define $\tilde
 U_\varepsilon(z):=\varepsilon^{-\gamma_s} U_j(\varepsilon z)$ as in
 Section \ref{sec:local-asympt-behav}.

We notice that $\tilde V_\varepsilon \in \mathcal{D}^{1,2}_{(\Omega/\varepsilon)^c}(\R^{N+1}_+;t^{1-2s})$, $\tilde V_\varepsilon-\tilde U_\varepsilon \in \mathcal{D}^{1,2}_{(\Omega/\varepsilon)^c\cup K}(\R^{N+1}_+;t^{1-2s})$ and
\begin{equation}\label{eq:tilde_V}
q(\tilde V_\varepsilon,\phi)=0 \quad\text{for all } \phi \in 
\mathcal{D}^{1,2}_{(\Omega/\varepsilon)^c\cup K}(\R^{N+1}_+;t^{1-2s}).
\end{equation}
In particular,
\begin{equation}\label{eq:cap_K_eps1}
\|\tilde V_\varepsilon \|^2_{\mathcal{D}^{1,2}(\R^{N+1};t^{1-2s})}
=\kappa_s \mathop{\rm Cap}\nolimits^s_{\Omega/\varepsilon}(K,\tilde u_\varepsilon),
\end{equation}
where $\tilde u_\varepsilon=\mathop{\rm Tr} \tilde U_\varepsilon$.

Let $r_0>0$ be such that
  $K\subset B_{r_0}'=\{x\in\R^N: |x|<r_0\}$.  For $\varepsilon$
sufficiently small, we have that
\[
B_{r_0}' \subset \frac{\Omega}{\varepsilon},
\]
so that $\mathcal{D}^{1,2}_{(B_{r_0}')^c\cup K} \subseteq \mathcal{D}^{1,2}_{(\Omega/\varepsilon)^c\cup K}$ and, in turn, 
\begin{equation}\label{eq:cap_K_eps2}
\mathop{\rm Cap}\nolimits^s_{\Omega/\varepsilon}(K,\tilde u_\varepsilon)
\leq \mathop{\rm Cap}\nolimits^s_{B_{r_0}'}(K,\tilde u_\varepsilon)
\to \mathop{\rm Cap}\nolimits^s_{B_{r_0}'}(K,\hat \psi)
\end{equation}
as $\varepsilon\to 0^+$, where in the last step we used \eqref{eq:17}
and Lemma \ref{lem:cap_f_n}. Combining \eqref{eq:cap_K_eps1} and
\eqref{eq:cap_K_eps2}, we deduce that the family
$\{\tilde V_\varepsilon\}_{\varepsilon>0}$ is bounded in the reflexive
space $\mathcal{D}^{1,2}(\R^{N+1}_+;t^{1-2s})$. Then there exist a sequence
$\varepsilon_n\to0^+$ and
$\tilde V\in \mathcal{D}^{1,2}(\R^{N+1}_+;t^{1-2s})$ such that
\begin{equation}\label{eq:weak_conv}
\tilde V_{\varepsilon_n}\rightharpoonup \tilde V \quad\text{weakly in }\mathcal{D}^{1,2}(\R^{N+1}_+;t^{1-2s})
\end{equation}
as $n\to+\infty$.

Let $\tilde\eta_K\in C^\infty_c(B_R^+)$ for some $R>0$, be such that
$\tilde\eta_k=1$ on a neighborhood of $K$. 
Then
$\tilde V_{\varepsilon_n}-\tilde\eta_K\tilde U_{\varepsilon_n} \in
\mathcal{D}^{1,2}_K(\R^{N+1}_+;t^{1-2s})
=\{U\in \mathcal{D}^{1,2}(\R^{N+1}_+;t^{1-2s}): \, 
\mathop{\rm Tr}U\in \mathcal D^{s,2}(\R^N\setminus K)\}$.
Moreover, by \eqref{eq:12} we have
that 
\begin{equation}\label{eq:20}
\tilde\eta_K\tilde U_\varepsilon\to 
\tilde\eta_K\tilde\psi \quad\text{in }\mathcal{D}^{1,2}(\R^{N+1}_+;t^{1-2s}).
\end{equation}
Since $\mathcal{D}^{1,2}_K(\R^{N+1}_+;t^{1-2s})$ is  closed in $\mathcal{D}^{1,2}(\R^{N+1}_+;t^{1-2s})$ (in the strong topology
and then, being a subspace, in the weak topology), 
 by \eqref{eq:weak_conv} we conclude that
$\tilde V-\tilde\eta_k\tilde\psi \in
\mathcal{D}^{1,2}_K(\R^{N+1}_+;t^{1-2s})$.
Moreover, relations \eqref{eq:tilde_V} and \eqref{eq:weak_conv}
provide
\[
q(\tilde V,\phi)=0 \quad \text{for all } \phi\in
C^\infty_c(\R^{N+1}_+\cup (\R^N\setminus K)),
\]
so that, by density,
\begin{equation*}
q(\tilde V,\phi)=0 \quad\text{for all } \phi \in \mathcal{D}^{1,2}_K(\R^{N+1}_+;t^{1-2s}).
\end{equation*}
In particular, 
\begin{equation}\label{eq:cap_K_eps3}
\|\tilde V\|^2_{\mathcal{D}^{1,2}(\R^{N+1}_+;t^{1-2s})}=
\kappa_s\mathop{\rm Cap}\nolimits^s_{\R^N}(K,\hat \psi)
=q(\tilde V,\tilde\eta_K\tilde\psi).
\end{equation}
Similarly, since $\tilde V_\varepsilon-\tilde\eta_K\tilde U_\varepsilon \in 
\mathcal{D}^{1,2}_{(\Omega/\varepsilon)^c\cup K}(\R^{N+1}_+;t^{1-2s})$
for $\varepsilon>0$ sufficiently small, using also relations \eqref{eq:tilde_V}, \eqref{eq:weak_conv}, \eqref{eq:20} and \eqref{eq:cap_K_eps3}, we obtain
\begin{equation}\label{eq:cap_K_eps4}
\|\tilde V_{\varepsilon_n}\|^2_{\mathcal{D}^{1,2}(\R^{N+1}_+;t^{1-2s})}
=q(\tilde V_{\varepsilon_n},\tilde\eta_K\tilde U_{\varepsilon_n})
\to q(\tilde V,\tilde \eta_K\tilde \psi)
=\kappa_s\mathop{\rm Cap}\nolimits^s_{\R^N}(K,\hat \psi),
\end{equation}
as $n\to+\infty$.
By the Urysohn's subsequence principle we conclude that
  the above convergence holds as $\eps\to0^+$ and not only along the
  sequence $\eps_n$.
To conclude the proof it suffices to notice that, by a change of variables,
\[
\mathop{\rm Cap}\nolimits^s_\Omega(K_\varepsilon,u_j)=
\frac1{\kappa_s}\|V_{\Omega,K_\varepsilon,
  u_j}\|^2_{\mathcal{D}^{1,2}(\R^{N+1}_+;t^{1-2s})}
=\frac1{\kappa_s}
\varepsilon^{N+2(\gamma_s-s)} \|\tilde V_\varepsilon\|^2_{\mathcal{D}^{1,2}(\R^{N+1}_+;t^{1-2s})}
\]
and to replace \eqref{eq:cap_K_eps4} into the previous expression.
\end{proof}

\begin{proof}[Proof of Theorem \ref{thm:main}]
The family of sets $\{\eps K\}_{\varepsilon>0}$ concentrates to the
compact set $\{0\}$, which satisfies $\mathop{\rm
Cap}\nolimits^s_\Omega(\{0\})=0$
by Example \ref{ex:singleton},  so that Theorem \ref{thm:expansion}
applies in our situation. By combining it with Theorem
\ref{thm:cap_K_eps}, we obtain the stated result.
\end{proof}

\begin{proof}[Proof of Corollary \ref{cor:sharp_asymp}]
Let $V_K$ be the function that achieves the infimum in
  \eqref{eq:16} with $u=\hat\psi$ and $\Omega=\R^N$, so that $\mathop{\rm Cap}\nolimits^s_{\R^N}(K,\hat\psi)=\frac1{\kappa_s}q(V_K,V_K)$. The Hardy-trace inequality \eqref{eq:half_space_hardy} provides
\begin{align*}
\mathop{\rm Cap}\nolimits^s_{\R^N}(K,\hat\psi)=\frac1{\kappa_s}q(V_K,V_K)
&\geq \Lambda_{N,s}\int_{\R^N} \frac{|\mathop{\rm Tr}V_K|^2}{|x|^{2s}}\,dx \\
&\geq \Lambda_{N,s}\int_{K} \frac{|\mathop{\rm Tr}V_K|^2}{|x|^{2s}}\,dx
= \Lambda_{N,s}\int_{K} |x|^{-2s} |\hat \psi(x)|^2 \,dx.
\end{align*}
If, by contradiction, $\mathop{\rm
    Cap}\nolimits^s_{\R^N}(K,\hat\psi)=0$ the above inequality would imply $\hat
  \psi=0$ a.e. in $K$. Since  the $N$-dimensional Lebesgue measure of
  $K$ is strictly positive
and $\hat\psi$ weakly solves $(-\Delta)^s\hat\psi=0$ in $\R^N$, the
Unique Continuation Principle from sets of positive measure proved in
\cite[Theorem 1.4]{FallFelli14}
would
imply that $\hat\psi\equiv 0$ in $\R^N$, giving rise to a
contradiction in view of Remark \ref{rem:hatpsinonzero}.
\end{proof}

\appendix 
\section{Boundedness of eigenfunctions}\label{sec:app_A}

To prove boundedness of eigenfunctions we need the following
Sobolev-trace inequality which follows from combination of Theorem
\ref{thm:immersions} (i) and continuity of the trace map
\eqref{eq:trace} (see also \cite[Theorem 2.1]{brandle2013concave}): there exists a positive
constant $\tau_{N,s}>0$ such that 
  \begin{equation}\label{eq:half_space_Sobolev}
   \tau_{N,s}  \|\mathop{\rm Tr}W\|_{L^{2^*\!(s)}(\R^N)}^2
        \leq \int_{\R^{N+1}_+}t^{1-2s}|\nabla
   W|^2\,dt\,dx,\quad\text{for all }W\in{\mathcal
     D}^{1,2}({\R^{N+1}_+};t^{1-2s}),
 \end{equation}
where $2^*\!(s)$ is defined in \eqref{eq:2*s}.
In the following lemma we prove that the extensions of eigenfunctions
of \eqref{eq:eigenvalue_main} are bounded in $\R^{N+1}_+$.
\begin{lemma}\label{l:bk}
Let $\Omega\subset\R^N$, $N\geq 1$, be a bounded open
set and $s\in (0,\min\{1,N/2\})$. 
Let $\alpha\in\R$ and $W\in \mathcal{D}^{1,2}_{\Omega^c}(\R^{N+1}_+;t^{1-2s})$ be a weak solution
to 
\begin{equation}\label{eq:7}
\begin{cases}
-\mathop{\rm div}(t^{1-2s} \nabla W)=0, &\text{in } \R^{N+1}_+, \\
\lim_{t\to0^+}\left(-t^{1-2s} \partial_t W \right)=\alpha W, &\text{in } \Omega\times\{0\}, \\
W=0, &\text{in } (\R^N\setminus \Omega)\times\{0\},
\end{cases}
\end{equation}
in the sense that 
\begin{equation}\label{eq:1}
\int_{\R^{N+1}_+} t^{1-2s}\nabla W\cdot\nabla\phi \,dx\,dt=\alpha\int_\Omega \mathop{\rm Tr}W\mathop{\rm Tr}\phi \,dx 
\end{equation}
for every $\phi \in\mathcal{D}^{1,2}_{\Omega^c}(\R^{N+1}_+;t^{1-2s})$.
Then $W\in L^\infty(\R^{N+1}_+)$ and 
$\mathop{\rm Tr}W\in L^\infty(\Omega)$.
\end{lemma}
\begin{proof}
The fact that $\mathop{\rm Tr}W\in L^\infty(\Omega)$ can be found in
\cite[Theorem 3.1, Remark 3.2]{BP}, see also
\cite{FranzinaPalatucci14}. Let us prove the statement about its
extension.
From the Poisson formula for problem \eqref{eq:7} given in
\cite{CaffarelliSilvestre07} we have that, for some constant
$C_{N,s}$,  
\[
W(x,t)=C_{N,s}\int_{\R^N}\frac{t^{2s}}{(|x-\xi|^2+t^2)^{\frac{N+2s}{2}}}\mathop{\rm
Tr}(W)(\xi)\,d\xi\quad\text{for all }(x,t)\in\R^{N+1}_+,
\]
hence
\begin{align*}
|W(x,t)|&\leq \|\mathop{\rm Tr}W\|_{L^\infty(\Omega)}
|C_{N,s}|\int_{\R^N}\frac{t^{2s}}{(|x-\xi|^2+t^2)^{\frac{N+2s}{2}}}\,d\xi\\
&=\|\mathop{\rm Tr}W\|_{L^\infty(\Omega)}
|C_{N,s}|\int_{\R^N}\frac{t^{2s}}{(|\xi|^2+t^2)^{\frac{N+2s}{2}}}\,d\xi\\
&=\|\mathop{\rm Tr}W\|_{L^\infty(\Omega)}
|C_{N,s}|\int_{\R^N}\frac{d\xi'}{(|\xi'|^2+1)^{\frac{N+2s}{2}}}\,d\xi'
\end{align*}
for all $(x,t)\in\R^{N+1}_+$, thus implying that $W\in L^\infty(\R^{N+1}_+)$ and completing the proof.
\end{proof}

\section{Fractional convergence of sets in the sense of Mosco}\label{sec:app_B}

We give the following definition which is the analogue of the standard sets convergence 
in the sense of Mosco (\cite{Mosco69}). 

\begin{definition}\label{def:fracmosco}
Let $\Omega\subset\R^N$ be a bounded open set.
Let $\{K_\varepsilon\}_{\varepsilon>0}$ be a family of compact sets contained in $\Omega$. 
We say that {\em $\{\Omega\setminus K_\eps\}_{\varepsilon>0}$ converges to $\Omega\setminus K$ in the fractional
  sense of Mosco} if the following two properties hold:
  \begin{enumerate}[(i)]
  \item the weak limit points in $\mathcal D^{s,2}(\R^N)$ of every
    family of functions
    $u_\eps\in \mathcal D^{s,2}(\Omega\setminus
    K_\eps)$
    belong to $\mathcal D^{s,2} (\Omega\setminus K)$;
  \item for every
    $u\in \mathcal D^{s,2} (\Omega\setminus K)$, there
    exists a family of functions
    $u_\eps\in \mathcal D^{s,2} (\Omega\setminus K_\eps)$ such that $u_\eps\to u$
    in $\mathcal D^{s,2}(\R^N)$.
\end{enumerate}
\end{definition}

In this appendix we prove that the notion of
  concentration introduced in Definition \ref{def:concentrates}
implies the convergence of
$\Omega\setminus K_\eps$ to $\Omega\setminus K$ in the fractional
sense of Mosco if
$\mathop{\rm Cap}\nolimits^s_\Omega(K)=0$. 

\begin{lemma}
  Let $\Omega\subset\R^N$ be a bounded open set and $K\subset\Omega$
  be a compact set with $\mathop{\rm Cap}\nolimits^s_\Omega(K)=0$.
  Let $\{K_\varepsilon\}_{\varepsilon>0}$ be a family of compact sets
  contained in $\Omega$ concentrating to $K$ 
 in the sense of Definition \ref{def:concentrates}. Then
  $\Omega\setminus K_\eps$ converges to $\Omega\setminus K$ in the
  fractional sense of Mosco as $\eps\to0^+$.
\end{lemma}
\begin{proof}
  We first prove that condition (i) in Definition \ref{def:fracmosco}
  is satisfied. Let us consider a family 
  $\{u_\eps\}_{\eps>0}\subset \mathcal
    D^{s,2}(\Omega\setminus K_\eps)$
   such that $u_\eps \rightharpoonup u $ in
  $\mathcal D^{s,2}(\R^N)$. We need to show that
  $u\in \mathcal D^{s,2}(\Omega\setminus K)$. Obviously
  $\{u_\eps\}_{\eps>0}\subset \mathcal
    D^{s,2}(\Omega)$
  and $\mathcal D^{s,2}(\Omega)$ is a closed subspace
  of $\mathcal D^{s,2}(\R^N)$.  Then $u\in \mathcal D^{s,2}(\Omega)$ since this
  space is closed in the weak topology. Furthermore, being
  $\mathop{\rm Cap}\nolimits^s_\Omega(K)=0$, Proposition
  \ref{prop:cap0} provides $\mathcal D^{s,2}(\Omega)=\mathcal D^{s,2}(\Omega\setminus K)$.
 
  We now address item (ii) in Definition \ref{def:fracmosco}. Let
  $u\in \mathcal D^{s,2}(\Omega \setminus K)$ and $U=\mathcal{H}(u)$ be its
  Caffarelli-Silvestre extension as in \eqref{eq:13}. We need to exhibit a sequence
  $u_\eps$ in $\mathcal D^{s,2}(\Omega\setminus K_\eps)$ which converges to $u$
  in $\mathcal D^{s,2}(\R^N)$. We note that for every $\delta>0$ there
  exists $\varphi_\delta\in C^\infty_c(\R^{N+1}_+\cup\Omega)$ such
  that
 \[
  \nor{\varphi_\delta - U }_{\mathcal D^{1,2}(\R^{N+1}_+; t^{1-2s})} < \delta.
 \]
Since by assumption $\mathop{\rm Cap}\nolimits^s_\Omega(K)=0$, then for every $n\in\N$ there exists $\eps_n>0$ 
and $\eta_n\in C^\infty_c(\overline{\R^{N+1}_+})$ such that 
$\{\eps_n\}$ is strictly decreasing to zero,
$\eta_n\equiv 0 $ in $\R^N\setminus \Omega$, 
$\eta_n\equiv 1 $ in a neighborhood of $K_{\eps}$ for all $\eps\in(0,\eps_n)$ and 
\[
 \int_{\R^{N+1}_+} t^{1-2s} |\nabla \eta_n|^2\,dx\,dt < \frac1n. 
\]
Let us define $W_{n} := \varphi_{\delta}(1-\eta_n)$. We note that
$W_n \in \mathcal D^{1,2}_{\Omega^c \cup K_\eps}(\R^{N+1}_+;t^{1-2s})$
for all $\eps\in(0,\eps_n)$.
Then, using \eqref{eq:8} 
we obtain
\begin{align*}
  \int_{\R^{N+1}_+} t^{1-2s}& |\nabla W_n- \nabla \varphi_{\delta}|^2 \,dx\,dt
  = \int_{\R^{N+1}_+} t^{1-2s} |\nabla ( \varphi_{\delta}\eta_n)|^2\,dx\,dt \\
  &\leq 2\int_{\R^{N+1}_+} t^{1-2s} |\varphi_{\delta}|^2|\nabla \eta_n
    |^2\,dx\,dt + 
    2\int_{\R^{N+1}_+} t^{1-2s}|\eta_n|^2 |\nabla \varphi_{\delta} |^2\,dx\,dt\\
  &\leq \frac{2\sup|\varphi_\delta|^2}{n}+
    2\left(\sup|\nabla\varphi_\delta|^2\right)\left(\sup_{z\in\mathop{\rm supp
    \varphi_\delta}}|z|^2\right)
    \int_{\R^{N+1}_+} t^{1-2s}\frac{|\eta_n|^2}{|x|^2+t^2}\,dx\,dt\\
 &\leq \frac{2\sup|\varphi_\delta|^2}{n}+\frac{8}{n(N-2s)^2}
    \left(\sup|\nabla\varphi_\delta|^2\right)\left(\sup_{z\in\mathop{\rm supp
    \varphi_\delta}}|z|^2\right).
\end{align*}
Hence there exists $n_\delta$ such that  
\[
\| W_n- \varphi_{\delta} \|_{\mathcal D^{1,2}(\R^{N+1}_+;
  t^{1-2s})} < \delta\quad\text{for all }n\geq n_\delta.
\]
For all $\eps\in(0,\eps_1)$ we let $U_\eps:=W_n$ where $n$ is such
that $\eps_{n+1}\leq\eps<\eps_n$. The above argument then yields that
$U_\eps\in \mathcal D^{1,2}_{\Omega^c \cup
  K_\eps}(\R^{N+1}_+;t^{1-2s})$ and
\[
\| U_\eps- \varphi_{\delta} \|_{\mathcal D^{1,2}(\R^{N+1}_+;
  t^{1-2s})} < \delta\quad\text{for all }\eps\in(0,\eps_{n_\delta}).
\]
Hence 
$
\| U_\eps- U\|_{\mathcal D^{1,2}(\R^{N+1}_+;
  t^{1-2s})}\leq 
\| U_\eps- \varphi_{\delta} \|_{\mathcal D^{1,2}(\R^{N+1}_+;
  t^{1-2s})}
+\| \varphi_{\delta} -U\|_{\mathcal D^{1,2}(\R^{N+1}_+;
  t^{1-2s})}
 < 2\delta$ for all $\eps\in(0,\eps_{n_\delta})$.

 We conclude that
 $U_\eps\to U$ in $\mathcal D^{1,2}(\R^{N+1}_+;
  t^{1-2s})$ and therefore $u_\eps=\mathop{\rm Tr}U_\eps\in
  \mathcal D^{s,2}(\Omega\setminus K_\eps)$ converges to $u =\mathop{\rm Tr}U$ in
 $\mathcal D^{s,2}(\R^N)$ by continuity of the trace map \eqref{eq:trace}.
\end{proof}

\def\cprime{$'$}


\end{document}